\documentclass[a4paper, 11pt]{scrartcl}
\usepackage[T1]{fontenc}
\usepackage[utf8]{inputenc}

\usepackage[english]{babel}
\usepackage{lmodern}
\usepackage{hyperref}
\hypersetup{colorlinks = true, citecolor = blue}
\usepackage{amsmath, amssymb, amsthm}
\usepackage{mathrsfs, stmaryrd, mathtools}
\usepackage{cleveref}
\theoremstyle{plain}
  \newtheorem{theorem}{Theorem}[section]
\theoremstyle{plain}
  \newtheorem{lemma}[theorem]{Lemma}
  
  \newtheorem{proposition}[theorem]{Proposition}
\theoremstyle{definition}
  \newtheorem{definition}{Definition}[section]
\theoremstyle{remark}
  
  \newtheorem{remark}{Remark}[section]
  \newtheorem{example}{Example}[section]

\crefname{lemma}{Lemma}{Lemmas}
\Crefname{lemma}{Lemma}{Lemmas}
\crefname{theorem}{Theorem}{Theorems}
\Crefname{theorem}{Theorem}{Theorems}
\crefname{corollary}{Corollary}{Corollaries}
\Crefname{corollary}{Corollary}{Corollaries}
\crefname{proposition}{Proposition}{Propositions}
\Crefname{proposition}{Proposition}{Propositions}
\crefname{section}{Section}{Sections}
\Crefname{section}{Section}{Sections}
\usepackage{tikz}
\usetikzlibrary{backgrounds}
\usepackage{pgfplots}
\usepgfplotslibrary{groupplots}

\newenvironment{keywords}{\par\noindent\emph{Keywords: }}{}
\newenvironment{amscat}{\par\noindent\emph{AMS 2010 subject
classifications:}}{}

\title{Invariant measures, Hausdorff dimension
  and dimension drop of some harmonic measures on Galton-Watson trees}
\author{Pierre Rousselin\thanks{LAGA, University Paris 13; Labex MME-DII}}
\date{\today}

\DeclarePairedDelimiterX\intff[2]{[}{]}{#1,#2}
\DeclarePairedDelimiterX\intfo[2]{[}{)}{#1,#2}
\DeclarePairedDelimiterX\intof[2]{(}{]}{#1,#2}
\DeclarePairedDelimiterX\intoo[2]{(}{)}{#1,#2}
\DeclarePairedDelimiterX\intint[2]{\llbracket}{\rrbracket}{#1,#2}
\newcommand\N{\mathbb{N}}
\newcommand\Ne{\N^{*}}
\newcommand\R{\mathbb{R}}

\newcommand\words{\mathcal{U}}
\newcommand\infinitewords{{\words_{\infty}}}
\renewcommand\root{\text{\o}}
\newcommand\prt[2][]{{#2}_{*#1}}
\newcommand\rays[1]{\partial{#1}}
\newcommand\cyl[2][]{\left[#2\right]_{#1}}
\newcommand\noleaf{\mathscr{T}}
\newcommand\markednoleaf{\mathscr{T}_m}
\newcommand\markednoleafwithrays{\mathscr{T}_{m,r}}
\newcommand\numch{\nu}
\newcommand\pref{\preceq}
\newcommand\gcp{\wedge}

\newcommand\harm{\mathsf{HARM}}
\newcommand\unif{\mathsf{UNIF}}
\newcommand\GW{\mathbf{GW}}
\newcommand\pp{\mathbf{p}}
\DeclarePairedDelimiter{\pars}{(}{)}
\DeclarePairedDelimiter{\bracks}{[}{]}
\DeclarePairedDelimiter{\abs}{\lvert}{\rvert}
\DeclarePairedDelimiterX{\setof}[2]{\lbrace}{\rbrace}{#1\,{:}\,#2}
\newcommand\dd{\mathop{}\!\mathrm{d}}
\newcommand\indic[1]{\mathbf{1}_{\left\lbrace#1\right\rbrace}}

\DeclareMathOperator\Gammafunc{\mathsf\Gamma} 
\renewcommand\P{\mathbb{P}}
\newcommand\E{\mathbb{E}}
\newcommand\et{\mathsf{et}} 
\newcommand\hausd{\dim_{\mathrm{H}}}
\DeclareMathOperator\dist{d}
\DeclareMathOperator\distrays{d_{\infinitewords}}
\DeclareMathOperator\distgamma{d^\Gamma}
\DeclareMathOperator\ball{\mathscr{B}}
\DeclareMathOperator\diam{diam}
\DeclareMathOperator\hausm{\mathscr{H}}
\renewcommand{\tilde}{\widetilde}

\begin{document}
\maketitle{}
\begin{abstract}
  We consider infinite Galton-Watson trees without leaves
  together with i.i.d.~random variables called marks on each
  of their vertices. We define a class of flow rules on marked
  Galton-Watson trees for which we are able, under some algebraic
  assumptions, to build explicit invariant measures.
  We apply this result, together with
  the ergodic theory on Galton-Watson trees developed in \cite{LPP95},
  to the computation of Hausdorff dimensions of harmonic measures
  in two cases. The first one is the harmonic measure of the (transient)
  $\lambda$-biased random walk on Galton-Watson trees, for which the
  invariant measure and the dimension were not explicitly known.
  The second case is a model of random walk on a
  Galton-Watson trees with random lengths for
  which we compute the dimensions of the harmonic measure and show
  dimension drop phenomenon
  for the natural metric on the boundary
  and another metric that depends on the random lengths.
\end{abstract}

\begin{keywords}
  Galton-Watson tree, random walk, harmonic measure, Hausdorff dimension,
  invariant measure, dimension drop.
\end{keywords}
\begin{amscat}
  Primary 60J50, 60J80, 37A50;
  secondary 60J05, 60J10.
\end{amscat}

\section{Introduction}
\label{sec:intro}

Consider an infinite rooted tree $t$ without leaves.
The boundary $\rays{t}$ of the tree $t$ is the set of all rays on
$t$, that is infinite nearest-neighbour paths on $t$ that start from the root
and never backtrack.
If we equip the space $\rays{t}$ with a suitable metric $\dist$,
we may consider the Hausdorff dimension of $\rays{t}$ with respect to $\dist$.
Now assume that we run a
\emph{transient} nearest-neighbour random walk
on the vertices of $t$, starting from the root of $t$.
For any height $n \geq 0$, let $\Xi_n$ be the \emph{last} vertex
at height $n$ that is visited by the random walk.
We call $\Xi$ the \emph{exit ray}
of this random walk. The distribution of $\Xi$, which is a probability measure
on the boundary $\rays{t}$ is called the \emph{harmonic measure} associated to
the random walk.
One may define the \emph{Hausdorff dimension of this measure} to be the minimal
Hausdorff dimension of a subset of $\rays{t}$ of full measure.
When this Hausdorff dimension is strictly less than the one of the
whole boundary, we say that the \emph{dimension drop} phenomenon occurs.
Informally, this means that almost surely, the exit ray may only
visit a very small part of the boundary $\rays{t}$.

This phenomenon was first observed by Makarov in \cite{makarov1985}
in the context of a two-dimensional Brownian motion hitting a Jordan curve.
On supercritical Galton-Watson trees, Lyons, Pemantle and Peres proved the
dimension drop for the simple random walk in the seminal paper~\cite{LPP95} and
for the (transient) $\lambda$-biased random walk in~\cite{LPP_biased}.
For a more concrete interpretation of this phenomenon, see
\cite[Corollary~5.3]{LPP_biased} and \cite[Theorem~9.9]{LPP95} which show that
with high probability, the random walk is confined to an exponentially smaller
subtree.
Another application of this result is given in \cite{blps2015}, where the
authors prove a cut-off phenomenon for the mixing time of the random walk on the
giant component of an Erd\H{o}s-Rényi random graph, started from a typical
point.

An analogous asymptotic result was discovered by Curien and Le Gall
in~\cite{Curien_LeGall_harmonic} for the simple random walk on critical
Galton-Watson trees conditioned to survive at height $n$, when $n$ is large.
It was then extended to the case where the reproduction law of the Galton-Watson
tree has infinite variance by Lin~\cite{Shen_harmonic_infinite_variance}.
In these works, the asymptotic result for finite trees is (not easily)
derived from the computation of the Hausdorff dimension associated to a random
walk on an infinite tree.
This infinite tree is the scaling limit of the reduced critical Galton-Watson
tree conditioned to survive at height $n$; it can be seen as a Galton-Watson
tree with edge lengths and the simple random walk becomes a nearest-neighbour
random walk with transition probabilities inversely proportional to the edge
lengths.

One of the main goals of this work is to generalize this last model and obtain
similar results.
This will be done in \Cref{sec:reclength}.
To introduce our method, we need a little more formalism.
We work in the space of \emph{marked} trees without leaves, that is, trees $t$
with, on each vertex $x$ of $t$, an attached non-negative real number
$\gamma_t(x)$ called the mark of $x$ in $t$.
A marked Galton-Watson tree, characterized by its reproduction law
and its mark law (which may be degenerated), is a Galton-Watson tree with
i.i.d.~marks.
To a positive function $\phi$ on the space of marked tree, we may associate the
law of a random ray $\Xi$ on the tree $t$ in the following way:
we first choose a vertex $\Xi_1 = i$ of height $1$ with probability
  \[
    \P(\Xi_1 = i) = \Theta_t(i)
    \coloneqq \frac{\phi(t[i])}{ \sum_{j=1}^{\numch_t(\root)}
    \phi (t[j])},
  \]
where $\numch_t(\root)$ is the number of children of the root $\root$ and $t[i]$
is the (reindexed) marked subtree starting from $i$, see \Cref{fig:flow_chain}.
For the remaining vertices of the ray $\Xi$, we then play the same game in the
selected subtree $t[i]$.  The sequence of such selected subtrees is a Markov
chain on the space of marked trees.
Our initial measure of interest for this Markov chain is the law of a marked
Galton-Watson tree but it is rarely invariant.

\begin{figure}[!ht]
\centering
\begin{tikzpicture}
  \node (root) at (0,0) {$\root, \gamma_t(\root)$};
  \node (i) at (0,3) {$i, \gamma_t (i)$};
  \node (1) at (-4,3) {$1, \gamma_t (1)$};
  \node (last) at (4,3) {$\nu_t(\root),\gamma_t(\nu_t(\root))$};
  \draw (root) -- (1);
  \draw (root) -- (i);
  \draw (root) -- (last);
  \node at (-2,3) {\dots};
  \node at (2,3) {\dots};

  \node (t1) [circle, draw=black] at (-4,6) {$t[1]$};
  \node (ti) [circle, draw=black] at (0,6) {$t[i]$};
  \node (tlast) [circle, draw=black] at (4,6) {$t[\nu_t(\root)]$};

  \draw (1) -- (-5,4);
  \node [above] at (-5,4) {\dots};
  \node at (-4,3.7) {\dots};
  \draw (1) -- (-3,4);
  \node [above] at (-3,4) {\dots};

  \draw (i) -- (-1,4);
  \node [above] at (-1,4) {\dots};
  \node at (0,3.7) {\dots};
  \draw (i) -- (1,4);
  \node [above] at (1,4) {\dots};

  \draw (last) -- (3,4);
  \node [above] at (3,4) {\dots};
  \node at (4,3.7) {\dots};
  \draw (last) -- (5,4);
  \node [above] at (5,4) {\dots};

  \node (t) [circle, draw=black] at (-3, -1) {$t$};
  \draw [-latex, dashed, out =120, in=200] (t) to (t1);
  \draw [-latex, dashed] (t) to [out = 90, in = 200] (ti);
  \draw [-latex, dashed] (t) to [out = 0, in = 10, looseness = 1.4]  (tlast);
  \node at (-4.2, 0) {$\Theta_t(1)$};
  \node at (-2.4, 0) {$\Theta_t(i)$};
  \node at (-1.4, -1.2) {$\Theta_t(\nu_T(\root))$};

  \begin{scope}[on background layer]
    \draw [dotted, fill=black!20] (-5.5,2.5) -- (-2.5,2.5) -- (-2.5,7) 
      -- (-5.5,7) -- cycle;
    \draw [dotted, fill=black!20] (-1.5,2.5) -- (1.5,2.5) -- (1.5,7) -- 
      (-1.5,7) -- cycle;
    \draw [dotted, fill=black!20] (2.5,2.5) -- (5.5,2.5) -- (5.5,7) -- 
      (2.5,7) -- cycle;
  \end{scope}
  \draw [dotted] (-7,-2) -- (7,-2) -- (7,7.2) -- (-7,7.2) -- cycle;
\end{tikzpicture}
\caption{Probability transitions of the Markov chain on the space of marked
trees associated to the flow rule $\Theta$.}
\label{fig:flow_chain}
\end{figure}
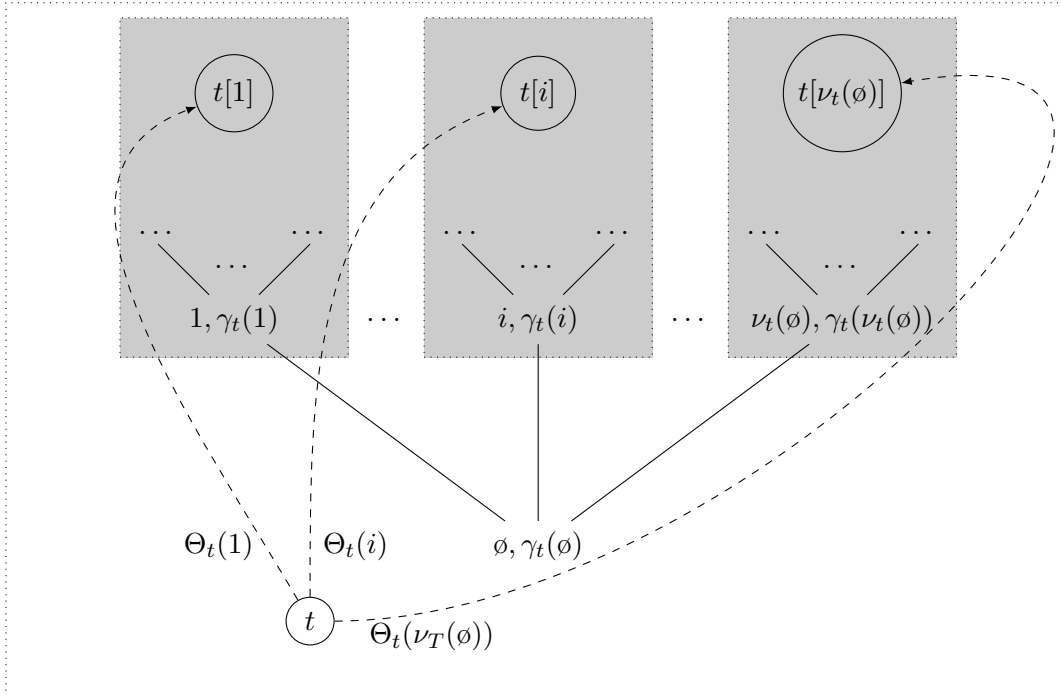

The heart of this work is \Cref{thm:alg}, in which we give some sufficient
algebraic conditions on the function $\phi$ to build an explicit invariant
measure for this Markov chain.  Moreover, this measure is absolutely continuous
with respect to the law of the Galton-Watson tree.
These algebraic conditions were inspired by the proofs of
\cite[Proposition~25]{Curien_LeGall_harmonic} and
\cite[Proposition~8]{Shen_harmonic_infinite_variance}.
This allows us to use the ergodic theory on Galton-Watson trees
from \cite{LPP95}, to obtain quantitative and qualitative results.

In \Cref{sec:lambda}, we use this tool to study the harmonic measure of the
$\lambda$-biased random walk on supercritical Galton-Watson trees.
In \cite{LPP_biased}, the authors prove
the existence of a unique invariant probability measure $\mu_{\harm}$ absolutely
continuous with respect to the law of the Galton-Watson tree, and conclude that
there is a dimension drop phenomenon.
However, they do not (except in the case $\lambda = 1$ in \cite{LPP95})
give an explicit formula for the density of $\mu_\harm$ and
the Hausdorff dimension of the harmonic measure.
Interestingly, although we are dealing with a model of random walk on
Galton-Watson trees without marks, it is easier to express an explicit
invariant measure in the more general setting of \Cref{thm:alg}.
From the explicit invariant measure, we deduce an expression for the Hausdorff
dimension of the harmonic measure in \Cref{thm:biased}.
This invariant measure and this Hausdorff dimension were found
independently by Lin in \cite{shen_lin_harmonic_biased}.
The explicit formulas allow him to answer interesting questions, including
an open question from
\cite{LPP97}.
In this work, we use our dimension formula \eqref{eq:dlambda}
to compute numerically the dimension
(and the speed, thanks to \cite{Elie_speed}) of the $\lambda$-biased random walk
as a function of $\lambda$, for two different reproduction laws.

\Cref{sec:reclength} is completely independent from~\Cref{sec:lambda}.
Let us briefly present the model we study in it.
Let $\Gamma$ be a random variable in $\intoo{1}{\infty}$. We consider a
Galton-Watson tree $T$ with reproduction law
$\pp = (p_k)_{k \geq 0}$ such that $p_0 = 0$ and $p_1 < 1$
and attach to its vertices i.i.d.~marks
$(\Gamma_x)_{x \in T}$
with the same law as $\Gamma$.
We add an artificial parent $\prt\root$ of the root $\root$.
Then we define the length of the edge between $\root$ and $\prt\root$
as the inverse $\Gamma_\root^{-1}$ of the mark of the root.
We do the same in all the subtrees starting from the children of the root,
except that we multiply all the lengths in these subtrees by
$ (1 - \Gamma_\root^{-1})$, and we continue recursively, see
\Cref{fig:recursive_lengths}.
We call the resulting random tree a Galton-Watson tree with recursive lengths.
The $\Gamma$-height of a vertex $x$ of $T$ is the sum of the lengths of all the
edges in the shortest path between $x$ and $\prt\root$.

\begin{figure}[!ht]
  \centering
\begin{tikzpicture}
  \draw [dashed] (-6,0) -- (6,0);
  \draw [dashed] (-6,10) -- (6,10);
  \draw [latex-latex, dotted] (-6,0) -- (-6,10);
  \node [left] at (-6,5) {1};
  \node [below] at (0,0) {$\prt\root$};
  \draw (0,0) -- (0,4);
  \node at (0,4) {$\bullet$};
  \node [below left] at (0,4) {$\root$};
  \draw [dotted, latex-latex] (0.2, 0) -- (0.2, 4);
  \node [right] at (0.2, 2) {$\Gamma_\root^{-1}$};
  \draw [dashed] (-5,4) -- (5,4);

  \draw (-5,4) -- (-5,7);
  \node at (-5,7) {$\bullet$};
  \node [below left] at (-5,7) {$1$};
  \draw [dotted, latex-latex] (-4.8, 4) -- (-4.8, 7);
  \node [right] at (-4.8, 5.5) {$\left(1 - \Gamma_\root^{-1}\right)\Gamma_1^{-1}$};

  \draw (-5.5, 7) -- (-5.5, 8);
  \node at (-5.5, 8) {$\bullet$};
  \node [below right] at (-5.5, 8) {$11$};
  \draw (-5.7, 8) -- (-5.7, 8.8);
  \node at (-5.7, 8.8) {$\bullet$};
  \draw (-5.3, 8) -- (-5.3, 9.3);
  \node at (-5.3, 9.3) {$\bullet$};
  \draw [dashed] (-5.7, 8) -- (-5.3, 8);

  \draw (-4.8, 7) -- (-4.8, 8.5);
  \node at (-4.8, 8.5) {$\bullet$};
  \node [below right] at (-4.8, 8.5) {$12$};
  \draw (-5, 8.5) -- (-5, 9.7);
  \node at (-5, 9.7) {$\bullet$};
  \draw (-4.6, 8.5) -- (-4.6, 9.1);
  \node at (-4.6, 9.1) {$\bullet$};
  \draw [dashed] (-4.6, 8.5) -- (-5, 8.5);
  
  \node [right] at (-4.5, 7.7) {\dots};

  \draw (-3.7, 7) -- (-3.7, 8.1);
  \node at (-3.7, 8.1) {$\bullet$};
  \node [below right] at (-3.7, 8.1) {$1\nu_T(1)$};

  \node at (-3.7, 9) {\dots};

  \draw [dashed] (-5.5,7) -- (-3.7,7);

  \draw (-1, 4) -- (-1, 6);
  \node at (-1,6) {$\bullet$};
  \draw [dotted, latex-latex] (-0.8, 4) -- (-.8, 6);
  \node [right] at (-.8, 5) {$\left(1 - \Gamma_\root^{-1}\right)\Gamma_2^{-1}$};
  \node [below left] at (-1, 6) {$2$};

  \draw [dashed] (-2, 6) -- (1, 6);

  \draw (-2,6) -- (-2, 7.3);
  \node at (-2, 7.3) {$\bullet$};
  \node [below left] at (-2, 7.3) {$21$};
  \node [above] at (-2,10) 
    {$ \left(1 - \Gamma_\root^{-1}\right)
      \left(1 - \Gamma_{2}^{-1}\right)
      \Gamma_{21}^{-1}$ };
  \draw [dotted, latex-latex] (-1.8, 6) -- (-1.8, 7.3);
  \draw [-latex] (-2,10.1) to [out = 330, in = 30] (-1.8, 6.65);

    \draw (-2.3,7.3) -- (-2.3, 8.3);
    \node at (-2.3, 8.3) {$\bullet$};
    \draw (-2.1,7.3) -- (-2.1, 9);
    \node at (-2.1, 9) {$\bullet$};
    \draw (-1.9,7.3) -- (-1.9, 7.8);
    \node at (-1.9, 7.8) {$\bullet$};
    \draw (-1.7,7.3) -- (-1.7, 8.1);
    \node at (-1.7, 8.1) {$\bullet$};
    \draw [dashed] (-1.7, 7.3) -- (-2.3, 7.3);

  \draw (-0.9,6) -- (-0.9, 6.7);
  \node at (-0.9, 6.7) {$\bullet$};
  \node [below right] at (-0.9, 6.7) {$22$};

  \node at (0,6.7) {\dots};

  \draw (1,6) -- (1, 7.1);
  \node at (1, 7.1) {$\bullet$};
  \node [above left] at (1, 7.1) {$2\nu_T(2)$};

  \node at (3, 5) {\dots};

  \draw (5,4) -- (5, 8);
  \node at (5, 8) {$\bullet$};
  \draw [dotted, latex-latex] (4.8, 4) -- (4.8, 8);
  \node [left] at (4.8, 6) {$\left( 1 -
    \Gamma_\root^{-1}\right)\Gamma_{\nu_T(\root)}^{-1}$};
  \node [below right] at (5,8) {$\nu_T(\root)$};

  \draw (4, 8) -- (4,9);
  \node at (4,9) {$\bullet$};
  \draw (3.8, 9) -- (3.8,9.5);
  \node at (3.8,9.5) {$\bullet$};
  \draw (4.2, 9) -- (4.2,9.7);
  \node at (4.2,9.7) {$\bullet$};
  \draw [dashed] (3.8,9) -- (4.2,9);

  \draw (4.7, 8) -- (4.7,8.5);
  \node at (4.7,8.5) {$\bullet$};
  \draw (4.5, 8.5) -- (4.5,9.1);
  \node at (4.5,9.1) {$\bullet$};
  \draw (4.9, 8.5) -- (4.9,9.5);
  \node at (4.9,9.5) {$\bullet$};
  \draw [dashed] (4.5,8.5) -- (4.9,8.5);
  \draw (5.4, 8) -- (5.4,9.2);
  \node at (5.4,9.2) {$\bullet$};
  \draw (6, 8) -- (6,8.7);
  \node at (6,8.7) {$\bullet$};
  \draw [dashed] (4,8) -- (6,8);

  \node at (1, 8) {\dots};
  
\end{tikzpicture}
\caption{A schematic representation of a Galton-Watson tree with recursive
lengths.}
\label{fig:recursive_lengths}
\end{figure}
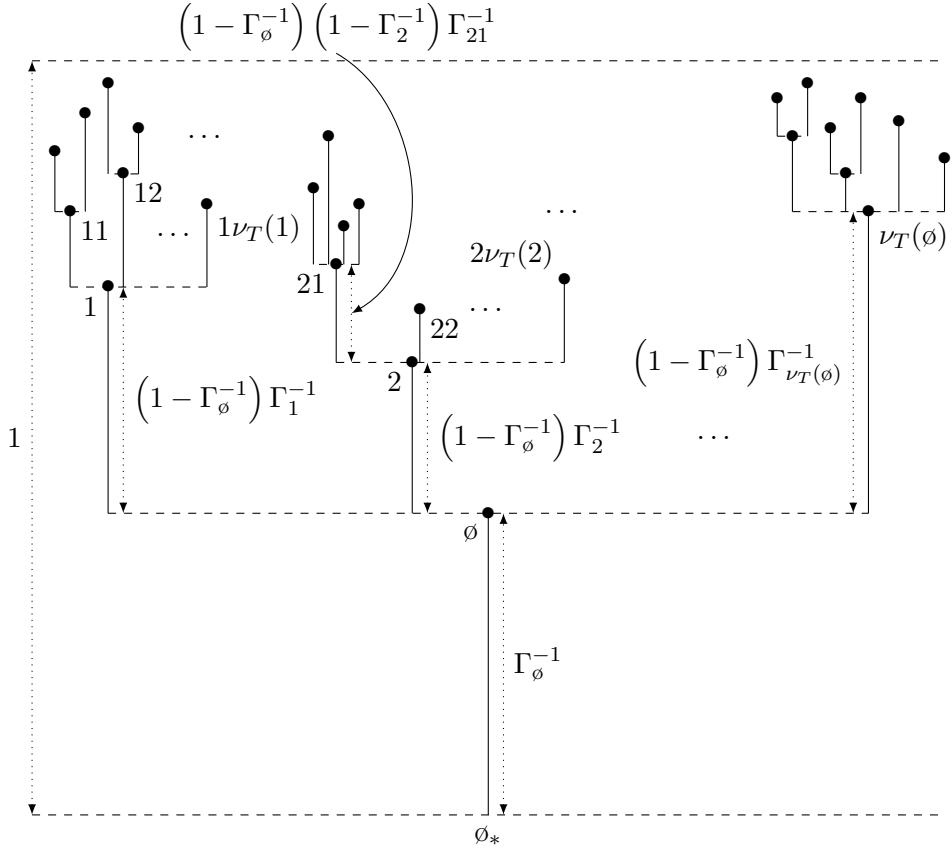

We then start a nearest-neighbour
random walk on $T$ from its root $\root$ with transition
probabilities inversely proportional to the lengths of the edges between the
current vertex and the neighbours.
The walk is reflected at $\prt\root$.
This random walk is almost surely transient, thus defines a harmonic measure on
the boundary $\rays{T}$.
We can again apply \Cref{thm:alg} to this model to find an invariant probability
measure.
This gives in \Cref{thm:reclengthd} the Hausdorff dimension of the harmonic
flow, as well as the dimension drop, for the so-called ``natural metric'' on the
boundary $\rays{T}$ which does not take into account the lengths on the tree,
only the genealogical structure.

Finally, we solve the same problem with respect to another metric on $\rays{T}$
which depends on the lengths.
We associate to our tree an \emph{age-dependent process} whose Malthusian
parameter turns out to be the Hausdorff dimension of the boundary with respect
to this distance, and prove that we still have a dimension drop phenomenon in
this context.

\paragraph{Acknowledgements.}
Part of this work was presented at the 47\textsuperscript{th} Saint-Flour Summer
School of Probability.
The author wishes to thank the organizers of this event.  The author is also
very grateful to his Ph.D.~supervisors Julien Barral and Yueyun Hu for many
interesting discussions and constant help and support and to the anonymous
referees for suggesting many improvements (and for the correction of an
uncountable amount of typos and other mistakes).
\section{Flow rules on marked Galton-Watson trees}
\label{sec:flowrules}
\subsection{Trees and flows}
Following Neveu (\cite{Neveu_gw}), we represent our trees as special subsets of
the finite words on the alphabet
$\Ne \coloneqq \{ 1, 2, \dotsc \}$.
The set of all finite words is denoted by $\words$
and equals the disjoints union
$\bigsqcup_{k = 0}^{\infty} (\Ne)^k$,
where we agree that
$(\Ne)^0 \coloneqq \{ \root \}$
is the set containing only the empty word $\root$.
The length of a word $x$ is the unique integer $k$ such that $x$ belongs to
$(\Ne)^k$ and is denoted by $\abs{x}$.
The concatenation of the words $x$ and $y$
is denoted by $x y$.
A word
$x = (x_1, x_2, \dots, x_{\abs{x}})$
is called a prefix of a word
$y = (y_1, y_2, \dots, y_{\abs{y}})$
when either $x = \root$ or
$\abs{x} \leq \abs{y}$ and $x_i = y_i$ for any $i \leq \abs{x}$.
We denote by
$\pref$ this partial order and by $x \gcp y$ the greatest common prefix of
$x$ and $y$.
The parent of a non-empty word
$x = (x_1, x_2, \dots, x_{\abs{x}})$
is
$\prt{x} \coloneqq (x_1 , x_2 , \dotsc , x_{\abs{x}-1})$ if
$\abs{x} \geq 2$;
otherwise it is the empty word $\root$.
We also say that $x$ is a child of $\prt{x}$.
A (rooted, planar, locally finite) tree $t$ is a subset of $\words$ such that
$\root \in t$ (in this context, we call $\root$ the root of $t$)
and for any $x \in t$:
\begin{itemize}
  \item if $x \neq \root$, then $\prt{x} \in t$;
  \item there exists a unique non-negative integer, denoted by $\numch_t (x)$
    and called the
    number of children of $x$ in $t$, such that for any $i \in \N$,
    $x i$ is in $t$ if and only if $1 \leq i \leq \numch_t(x)$.
\end{itemize}
The tree $t$ is endowed with the undirected graph structure obtained by drawing
an edge between each vertex and its children.
A leaf of $t$ is a vertex that has no child. In this work, we are interested in
leafless trees, that is infinite trees in which the number of children cannot be
$0$. We denote by $\noleaf$ the set of all such trees.

We will need to add an artificial parent $\prt\root$ of the root.
Let $\prt\words$ be the set
$\words \sqcup \{ \prt\root \}$ and let us agree
that $\prt\root$ is the only strict prefix of $\root$ and has length $-1$.
Likewise, for
any tree $t$ in $\noleaf$, we let
$\prt{t} \coloneqq t \sqcup \{ \prt\root \}$ and denote by
$\prt\noleaf$ the set of all such trees.

Let $\infinitewords \coloneqq (\Ne)^{\Ne}$ be the set of all infinite words.
The $k$-th truncation of an infinite word $\xi$ is the finite word composed of
its $k$ first letters and is denoted by $\xi_k$.
When a finite word $u$ is a truncation of an infinite word
$\xi$, we still say that $u$ is a prefix of $\xi$.
For two distinct infinite words $\xi$ and $\eta$, we may again consider their
greatest common prefix $\xi \wedge \eta \in \words$.
A ray in a tree $t$ is an infinite word $\xi$ such that each of its truncations
belongs to $t$.
One may also think of a ray as an infinite non-backtracking path in $t$ starting
from the root.
The set of all rays in $t$ is called the boundary of $t$ and denoted by
$\rays{t}$.

Let us denote by $\cyl[t]{x}$ the cylinder of all rays passing through $x$ in
$t$, that is the set of all rays $\xi$ in $\rays{t}$ such that
$\xi_{\abs{x}} = x$.
The boundary $\partial t$ of a tree $t$ is always endowed with the topology
generated by all the cylinders $\cyl[t]{x}$, for $x$ in $t$.

A (unit) \emph{flow} on $t$ is a function $\theta$ from $t$ to $\intff{0}{1}$, such
that $\theta(\root) = 1$ and for any $x \in t$,
\[
  \theta (x) = \sum_{i=1}^{\nu_t(x)} \theta(xi).
\]
We may define a flow $\theta_M$ on $t$ from a Borel probability measure $M$ on
the boundary $\rays{t}$ by setting $\theta_M (x) = M (\cyl[t]{x})$.
A monotone class argument shows that the mapping $M \mapsto \theta_M$ is a
one-to-one correspondance and we will abuse notations and write $\theta$ for
both the flow on $t$ and the associated Borel probability measure on $\rays{t}$.

\subsection{Hausdorff measures and dimensions on the boundary of a tree}
Consider an infinite leafless rooted tree $t$ as before, and a metric $\dist$ on
its boundary.
For $\xi$ in $\rays t$ and a non-negative number $r$, let $\ball(\xi,r)$ be the
closed ball centered at $\xi$ with radius $r$ for the metric $\dist$.
We always make the assumption that the balls of the metric space
$(\rays t, \dist)$ are exactly the cylinders $\cyl[t]{x}$, for $x$ in $t$.
An example of a metric satisfying this assumption is the \emph{natural distance}
$\distrays$ defined by,
\begin{equation}\label{eq:defdistnat}
  \forall \eta \neq \xi \in \rays t, \quad
  \distrays \left(\xi, \eta\right)
  = e^{-\abs{\xi\gcp\eta}}.
\end{equation}
We will use this metric in \Cref{sec:lambda} and the first half of
\Cref{sec:reclength}.
At the end of this paper, we will need to consider another (random) metric.

For $\delta > 0$, and a subset $E$ of $\rays t$, a $\delta$-cover of $E$ is a
denumerable family $\left(E_i\right)_{i\geq 1}$ of subsets of $\rays t$ such
that for any $i \geq 1$, the diameter $\diam E_i$ (with respect to $\dist$)
is lesser or equal to $\delta$
and $E \subset \bigcup_{i\geq 1} E_i$.
Now, let $s$ be a non-negative real number. First we define the (outer) measure
$\hausm_\delta^s$ by
\begin{equation}\label{eq:hausmdeltas}
  \hausm^s_\delta (E)
  =
  \inf
  \setof[\Big]{ \sum_{i=1}^{\infty} \diam(E_i)^{s}}{
    \left(E_i\right)_{i\geq 1} \text{is a $\delta$-cover of $E$}
  }.
\end{equation}
The $s$-dimensional Hausdorff measure of $E$ (with respect to $d$) is
\begin{equation}\label{eq:hausms}
  \hausm^s (E) =
  \lim_{\delta \to 0} \hausm^{s}_\delta(E)
  \in \intff{0}{\infty}.
\end{equation}
As usual, the Hausdorff dimension of a subset $E$ of $\rays t$ is
\[
  \hausd E
  = \inf\setof{ s \geq 0 }{\hausm^s(E) = 0}
  = \sup\setof{ s \geq 0 }{\hausm^s(E) = +\infty}.
\]

The (upper) Hausdorff dimension of a flow $\theta$
on the leafless tree $t$ is defined by
\[
  \hausd \theta \coloneqq \inf
  \setof{\dim_H (B) }{B \text{ Borel subset of } \rays{t}, \, \theta(B) = 1}.
\]
In this work, we will mainly deal with well-behaved flows that satisfy the
following (strong) property, which allows us to compute their (Hausdorff) dimensions.

\begin{definition}
  Let $\theta$ be a Borel probability measure on $(\rays t, \dist)$.
  If there exists
  $d_\theta$ in $\intff{0}{\infty}$ such that, for $\theta$-almost every
  $\xi$ in $\rays t$, the limit
  \[
    \lim_{r \downarrow 0}
    \frac{-\log\theta \ball (\xi,r)}{-\log r}
  \]
  exists and equals $d_\theta$,
  one says that the flow $\theta$ is exact-dimensional, of dimension
  $\dim^{\dist}
  \theta= d_\theta$.
\end{definition}
It is well-known in the Euclidean space that, when a Borel probability measure
$\mu$ is exact-dimensional of dimension $d_\mu$,
then all reasonable notions of dimensions coincide and are equal to $d_\mu$
(see~\cite[Theorem~4.4]{young_dimension}).
In particular, the (upper) Hausdorff dimension of $\mu$ equals its dimension.
On the boundary of a tree, with our assumption on the metric $\dist$,
this is also true (it is in fact easier than the Euclidean case).
The interested reader may want to read \cite[Section~15.4]{LyonsPeres_book} for
a simplification of \eqref{eq:hausmdeltas} on the boundary of a tree,
then
\cite{trico_straymond88} and
\cite{cutler_density} for density theorems (in \cite{trico_straymond88},
the authors consider centered covering measures as an alternative to Hausdorff
measures, but in our case, they are the same), and finally,
\cite[Proposition~10.2]{falconer_techniques}, for how to relate local and global
dimensions of a measure, using density theorems.
\begin{example}
  For the the natural distance $\distrays$, a flow $\theta$ on $t$
  is exact-dimensional of dimension $d_\theta$
  if and only if
  \begin{equation}\label{eq:hausdthetanatural}
    \text{for $\theta$-almost every $\xi$,}\quad
    \lim_{n \to \infty} \frac{1}{n} (-\log) (\theta(\xi_n))
    = d_\theta.
  \end{equation}
\end{example}

\subsection{Marked trees and flow rules}
We now introduce the marks, which can be thought of as additional data on the
vertices of the tree, in the form of non-negative real numbers.
A (leafless) \emph{marked tree} is a tree
$t \in \noleaf$ together with a collection
$\left( \gamma_t(x) \right)_{x \in t}$ of
elements of $\R_+$.

We define the (local) distance between two marked trees $t$ and $t'$ by
\[
  \dist_{m} \left(t, t'\right) =
  \sum_{n \geq 0}^{} 2^{-r-1} \delta_{n} \left(t, t'\right),
\]
where $\delta_n$ is defined by
\[
  \delta_n \left(t, t'\right)
  =
  \left\lbrace
  \begin{array}{l}
    1 \text{ if $t$ and $t'$ (without their marks)
      do not agree up to height $n$};
    \\
    \min
    \left(1,
      \sup
      \setof{\abs{\gamma_t (x) - \gamma_{t'} (x)}}{
        x \in t, \, \abs{x} \leq n
      }
    \right)
    \text{ otherwise.}
  \end{array}
  \right.
\]
It is a Polish space, which we denote by
$\markednoleaf$ and equip with its Borel $\sigma$-algebra.

For a marked tree $t$ and a vertex $x \in t$, we let
\[
  t [x] \coloneqq \setof{u \in \words}{x u \in t}
\]
be the reindexed subtree starting from $x$ with marks
\[
  \gamma_{t[x]} (y) = \gamma_t (x y), \quad \forall y \in t[x].
\]

A (consistent) \emph{flow rule} is a
measurable function $\Theta$ on a Borel subset $B$
of $\markednoleaf$ such that for any tree $t$ in $B$,
\begin{enumerate}
  \item $\Theta_t$ is a flow on $t$;
  \item for any vertex $x$ in $t$, $\Theta_t (x) > 0$;
  \item for all $x$ in $t$, $t[x]$ is in $B$ and for all
    $y$ in $t[x]$,
    $\Theta_{t[x]} (y) = \Theta_{t}(x y) / \Theta_t(x)$.
\end{enumerate}
Let $\mu$ be a Borel probability measure on $\markednoleaf$.
We say that $\Theta$ is a $\mu$-flow rule if $\mu(B) = 1$.
When the $\mu$-flow rule does not depend on the marks, it is essentially the same
notion of flow rule as in \cite{LPP95} and \cite[chapter~17]{LyonsPeres_book},
except that we allow it to only be defined on a set of full measure and
that we impose that $\Theta_t(x) > 0$ for any tree $t$ and any vertex
$x$ in $t$.
These choices were made in order to simplify the statements of the next
subsection.

The set of all marked trees with a distinguished ray is denoted by
\[
  \markednoleafwithrays \coloneqq
  \setof{ \left(t, \xi\right) }{t \in \markednoleaf, \, \xi \in \rays{t}}.
\]
It is endowed with the same distance as before except that for $n \geq 0$, we
impose that
$\delta_n \left(  \left(t,\xi\right), \left(t', \xi'\right)\right)$ is equal to
$1$ if the two distinguished rays
$\xi$ and $\xi'$ do not agree up to height $n$.
It is again a Polish space.
The \emph{shift} $S$ on this space is defined by
\[
  S ( t, \xi ) = (t[\xi_1], S\xi ),
\]
where
$S\xi$ is the infinite word obtained by removing the first letter of $\xi$.

Let $\mu$ be a Borel probability measure on $\markednoleaf$ and $\Theta$ be a
$\mu$-flow rule. We build a probability measure on $\markednoleafwithrays$,
denoted by
$\mu \ltimes \Theta$, or $\E_\mu$ when the $\mu$-flow rule $\Theta$ is fixed,
by first picking a tree $T$ at random according to $\mu$,
and then choosing a ray $\Xi$ of $T$ according to the distribution $\Theta_T$.

The sequence of trees $(T [\Xi_n] )_{n \geq 0}$ is then a
discrete time Markov chain with values in $\markednoleaf$, initial distribution
$\mu$ and transition kernel $P_\Theta$ given by
\[
  P_\Theta (T, T') =
  \sum_{i=1}^{\numch_T(\root)} \Theta_T(i) \indic{T' = T[i]}.
\]
If the law of $T[\Xi_1]$ is still $\mu$, we say that $\mu$
is a $\Theta$-invariant measure.
When $\mu$ is $\Theta$-invariant (for the Markov chain), the system
$ (\markednoleafwithrays, \mu \ltimes \Theta, S)$ is
measure-preserving.

\subsection{Ergodic theory on marked Galton-Watson trees}

We now consider a reproduction law $\pp = (p_k)_{k \geq 0}$, that is
a sequence of non-negative real numbers such that
$\sum_{k=0}^{\infty} p_k = 1$,
together with a random variable $\Gamma$ with values in $\R_+$.
We will assume throughout this work that $p_0 = 0$, $p_1 < 1$ and
the expectation
$m \coloneqq \sum_{k=0}^{\infty}kp_k$
is finite.

A $(\Gamma, \pp )$-Galton-Watson tree is a random tree $T$ such that
the root $\root$ has $k$ children with probability $p_k$, has a mark
$\gamma_T(\root) = \Gamma_\root$ with the same law as $\Gamma$ and,
conditionally on the event that the
root has $k$ children, the trees $T[1]$, $T[2]$, \dots,
$T[k]$ are independent and are $(\Gamma, \pp)$-Galton-Watson trees.
In particular, all the marks are i.i.d. We denote them
by $ (\Gamma_x)_{x \in T}$.
The distribution of a $(\Gamma, \pp)$-Galton-Watson on
$\markednoleaf$ is denoted by $\GW$.
The branching property is still valid in this setting of marked trees: if $T$ is
a random marked tree distributed as $\GW$, then for all integers $k \geq 1$,
Borel sets $A_1$, $A_2$, \dots, $A_k$ of $\markednoleaf$, and Borel sets $B$
of $\R_+$,
\[
  \P \bigl(
    \nu_T (\root) = k, \, \Gamma_\root \in B, \,
    T[1] \in A_1, \dotsc, T [k] \in A_k
  \bigr)
  =
  p_k \P \bigl(\Gamma \in B \bigr)
  \prod_{i=1}^{k} \GW \bigl(A_i\bigr).
\]

We now recall the main ingredients (all of which will be used throughout this
work) of the ergodic theory on Galton-Watson trees developed in
\cite[Section~5]{LPP95} (see also Section~17.5 of \cite{LyonsPeres_book}).
Since our marked Galton-Watson trees still satisfy the branching property, all
these results still apply in our setting of marked trees.
Remark that in the two following statements, the omission of the extra
hypotheses in \cite[Section~5]{LPP95} is due to our slightly modified
definition of a flow rule (we impose that for $\GW$-almost any tree $t$, we have
$\Theta_t(x) > 0$ for all $x$ in $t$).
\begin{lemma}[{\cite[Proposition~5.1]{LPP95}}]
  \label{prop:differentflow} 
  Let $\Theta$ and $\Theta'$ be two $\GW$-flow rule. Then
  $\Theta_T$ and $\Theta'_T$ are either almost surely equal or almost surely different.
\end{lemma}

\begin{theorem}
  [{\cite[Proposition~5.2]{LPP95}}]
  \label{prop:flow_fondamental}
  Assume there exists a $\Theta$-\emph{stationary} probability measure
  $\mu_{\Theta}$ which is
  absolutely continuous with respect to $\GW$. Then $\mu_{\Theta}$ is
  equivalent to $\GW$ and the associated measure-preserving system
  is ergodic. Moreover, the probability measure $\mu_\Theta$ is the only
  $\Theta$-stationary probability
  measure which is absolutely continuous with respect to $\GW$.
  The flow $\Theta_T$ is almost surely exact-dimensional, and its dimension,
  with respect to the metric $\distrays$ is
  \begin{equation}\label{cor:hausdflow}
    \dim^{\distrays} \Theta_T =
    \E_{\mu_\Theta} \bracks*{ -\log \Theta_T \left(\Xi_1\right) }.
  \end{equation}
\end{theorem}

The ergodic theory Lemma~{6.2} in \cite{LPP95} (see also
\cite[Lemma~17.20]{LyonsPeres_book}) will be used several times.
We recall it for the reader's convenience.
\begin{lemma}\label{lem:classic}
  If $S$ is a measure-preserving transformation on a probability space, $g$ is
  a finite and measurable function from this space to $\R$,
  and $g - Sg$ is bounded from below by an integrable
  function, then $g - Sg$ is integrable with integral $0$.
\end{lemma}

The uniform flow rule $\unif$ is defined as in \cite{LPP95}~Section~6; it does
not depend on the marks. For any tree $t$, we denote
\[
  Z_n (t) \coloneqq \# \setof{x \in t}{\abs{x} = n}.
\]
Let $T$ be a $(\Gamma,\pp)$-Galton-Watson tree.
By the Seneta-Heyde theorem, there exists a sequence of positive real numbers
$(c_n)_{n \geq 0}$
which only depends on $\pp$ such that almost surely
\[
  \tilde{W} \left(T \right) \coloneqq
  \lim_{n \to \infty} \frac{Z_n (T)}{c_n}
\]
exists in $\intoo{0}{\infty}$. Furthermore,
$\lim_{n \to \infty} c_{n+1}/c_n = m$.
Hence we may define the $\GW$-flow rule $\unif$ by
\[
  \unif_T (i) =
  \frac{
    \tilde{W} (T[i])}{
    \sum_{j=1}^{\nu_t(\root)} \tilde{W} (T[j])
  },
  \quad \text{for } j = 1,2,\dotsc,\nu_t(\root)
\]
and $\tilde{W}$ almost surely
satisfies the recursive equation
\begin{equation}\label{eq:recunif}
  \tilde{W}(T) = \frac{1}{m}
    \sum_{j=1}^{\nu_t(\root)} \tilde{W} (T[i]).
\end{equation}

Theorem~{7.1} in \cite{LPP95} is still valid with the same proof in our
setting of marked trees.
\begin{proposition}[Dimension drop for non uniform flow rules,
  {\cite[Theorem~7.1]{LPP95}}]
  \label{prop:dimensiondrop}
  With the same hypotheses as \Cref{prop:flow_fondamental}, if
  $\P \pars*{\Theta_T = \unif_T} < 1$,
  then, almost surely $\dim \Theta_T$ is strictly less
  than $\log m$, which equals $\hausd \rays{T}$.
\end{proposition}

\section{Invariant measures for a class of flow rules}
\label{sec:invariant}
Let $T$ be a $\left(\Gamma, \pp\right)$-Galton-Watson tree.
Let $\phi$ be a measurable positive function on a Borel set
of $\markednoleaf$ which has full $\GW$-measure.
We define a $\GW$-flow rule $\Theta$ by
\[
  \Theta_T ( i )
  =
  \frac{
    \phi (T[i])}{
    \sum_{j=1}^{\nu_T(\root)} \phi(T[j])
  },
  \quad \forall 1 \leq i \leq \numch_T(\root).
\]
\begin{lemma}\label{lem:equalflows}
  Let $\Theta$ and $\Theta'$ be two flow rules defined respectively by $\phi$
  and $\phi'$ as above. Then, $\Theta_T = \Theta'_T$ almost surely if and only
  if the functions $\phi$ and $\phi'$ are proportional.
\end{lemma}
\begin{proof}
  This is the same proof (except the last sentence) as
  \cite[Proposition~8.3]{LPP95}, so we omit it.
\end{proof}
We now assume that the marks and the function $\phi$ have their values in the
same sub-interval $J$ of $\intoo{0}{\infty}$ and that there exists a measurable
function $h$ from $J \times J$ to $J$ such that, almost surely,
\[
  \phi (T) = h \pars[\Big]{
    \Gamma_\root, \sum_{i=1}^{\nu_t (\root)} \phi (T[i])
  }.
\]
In words, $\phi$ is an observation on the tree $T$ that can be recovered from
the mark of the root $\Gamma_\root$ and the sum of such observations on the
subtrees $T[1]$, \dots, $T[\nu_T (\root)]$.

We now make algebraic assumptions on the function $h$.
These assumptions, as well as the next theorem,
are inspired by the proofs of
\cite[Proposition~25]{Curien_LeGall_harmonic} and
\cite[Proposition~8]{Shen_harmonic_infinite_variance}.
\begin{description}
  \item[symmetry] $\forall u, v \in J, \ h(u,v) = h(v,u)$;
  \item[associativity] $\forall u,v,w \in J, \ h(h(u,v),w) = h(u, h(v,w))$;
  \item[position of summand]
    $
      \forall u,v \in J,\ \forall a > 0, \
      \displaystyle\frac{h(u+a, v)}{(u+a)v} = \frac{h(u,a + v)}{u(a+v)}.
    $
\end{description}
Here are examples of such functions.
\begin{enumerate}
  \item $J = \intoo{0}{\infty}$ and $h(u,v) = \alpha uv$,
    for some $\alpha > 0$ ;
  \item for $c > 0$, $J = \intoo{c}{\infty}$ and $h(u,v) = \frac{uv}{u+v-c}$ ;
  \item for $d \geq 0$, $J = \intoo{0}{\infty}$
    and $h(u,v) = \frac{uv}{u+v +d}$.
\end{enumerate}
We treat the second case. By writing
\[
  h(u,v) = \left(u^{-1} + v^{-1} -c u^{-1}v^{-1} \right)^{-1}
  = \left(u^{-1} \left(1 -cv^{-1}\right) + v^{-1}\right)^{-1},
\]
and noticing that $ (1 - cv^{-1}) > 0$, we see that
\[
  h(u,v) > \left( c^{-1} \left(1 - cv^{-1}\right) +
  v^{-1}\right)^{-1} > c.
\]
Symmetry is clear, so is the last property because $h(u,v) / (uv) $ only
depends on the sum of $u$ and $v$.
Associativity follows from the following identity :
\begin{align*}
  &h( h(u,v), w)
  =
  \left( \left(u^{-1} + v^{-1} - c u^{-1}v^{-1}\right)
    + w^{-1} -c \left(u^{-1} + v^{-1} - c u^{-1}v^{-1}\right) w^{-1}
  \right)^{-1} \\
  &=
  \left( u^{-1} + v^{-1} + w^{-1} -c \left(u^{-1}v^{-1}
  + u^{-1}w^{-1} + v^{-1}w^{-1}\right)
  + c^2 u^{-1}v^{-1}w^{-1}
  \right)^{-1}.
\end{align*}

We will use examples~2 and {3} in the next section, with deterministic marks
equal to $1$ and example~{2} in \Cref{sec:reclength} with random marks.
The first example is illustrated by $\unif$, when $\alpha = 1/m$ and the marks
are all equal to one.
Another example with random marks is given by the flow rule
$\unif^\Gamma$ in \Cref{prop:dimunifgamma}.

For $u$ in $J$, define
\[
  \kappa(u) =
  \E \bracks[\bigg]{
    h\pars[\Big]{
      u,
      \sum\nolimits_{i=1}^{\numch_{\tilde{T}}(\root)} \phi (\tilde{T}[i])
    }
  },
\]
where $\tilde{T}$ is a
$(\Gamma,\pp)$-Galton-Watson tree independent of $T$.

\begin{theorem}\label{thm:alg}
  Assume that
  $C \coloneqq \E [\kappa (\phi (T))] < \infty$.
  Then the probability measure $\mu$ with density
  $C^{-1} \kappa (\phi(T))$
  with respect to $\GW$ is $\Theta$-invariant.
\end{theorem}
\begin{proof}
  We use the same notations as in the previous discussion. Let $f$ be a
  non-negative measurable function on the set of marked trees
  and let $\Xi$ be a random ray on $T$ distributed according to
  $\Theta_T$. We need to show that
  \[
    \E \bracks*{ f\pars*{T[\Xi_1]} \kappa (\phi\left(T\right) )}
    =
    \E \bracks*{f (T) \kappa (\phi(T))}.
  \]
  We compute the left-hand side.
  By conditioning on the value of $\numch_T (\root)$, we get
  \[
    \E [ f (T [\Xi_1] ) \kappa (\phi \left(T\right))]
    =
    \sum_{k = 1}^{\infty} p_k
      \sum_{i = 1}^{k}
      \E \bracks*{f \left(T[i]\right) \indic{\Xi_1 = i}
        \kappa \left(\phi(T)\right)}.
  \]
  Since the other terms only depend on $T$, we can replace $\indic{\Xi_1 = i}$
  by its conditional expectation given $T$, which is
  \[
    \frac{ \phi (T [i] )}{ \sum_{j=1}^{k} \phi ( T [j] ) }.
  \]
  The function $\kappa (\phi (T) )$ being also symmetrical in $T[1]$,
  $T[2]$, \dots, $T[k]$, all the terms in the sum are equal:
  \[
    \E \bracks*{ f ( T [\Xi_1] ) \kappa (\phi (T) ) }
      = \sum_{k= 1}^{\infty} p_k k
      \E \bracks[\bigg]{
        f \left(T[1]\right)
        \frac{\phi(T[1])}{ \sum_{j=1}^{k} \phi (T [j] )}
        \kappa \pars[\Big]{
          h \pars[\big]{
            \Gamma_\root,
            \sum_{j=1}^{k} \phi (T [j] )
          }
        }
      }.
  \]
  The definition of $\kappa$ and Tonelli's theorem give
  \begin{align}\label{eq:hh}
    &\E [ f ( T [\Xi_1] ) \kappa (\phi \left(T\right)) ]
    \nonumber \\
    &= \sum_{k= 1}^{\infty} p_k k
    \E \bracks*{
      f (T[1])
        \frac{\phi(T[1])}{ \sum\nolimits_{j=1}^{k} \phi (T [j] )}
        h \pars*{
          h \pars*{
            \Gamma_\root,
            \sum\nolimits_{j=1}^{k} \phi (T [j] )
          },
          \sum\nolimits_{r=1}^{\nu_{\tilde{T}}(\root)} \phi(\tilde{T}[r])
        }
    }.
  \end{align}
  We now use the assumptions on $h$. We first use symmetry and
  associativity to obtain
  \begin{align*}
    &h \pars*{
      h \pars*{
        \Gamma_\root,
        \sum\nolimits_{j=1}^{k} \phi (T [j] )
      },
      \sum\nolimits_{r=1}^{\nu_{\tilde{T}} (\root)}
      \phi(\tilde{T}[r])
    }
    \\
    &=
    h \pars*{
      h \pars*{
        \Gamma_\root,
        \sum\nolimits_{r=1}^{\nu_{\tilde{T}} (\root)} \phi(\tilde{T}[r])
      },
      \phi (T[1]) +
      \sum\nolimits_{j=2}^{k} \phi (T [j] )
    }.
  \end{align*}
  Since we never used the value of the mark of the
  root of $\tilde{T}$, we might as well decide that it is $\Gamma_\root$, so that
  \[
    h \pars*{
      \Gamma_\root,
      \sum\nolimits_{r=1}^{\nu_{\tilde{T}}(\root)}
      \phi (\tilde{T}[r])
    }
    = \phi (\tilde{T}).
  \]
  Notice that $\tilde{T}$ has the same law as $T$ and
  is independent of $T[1]$, $T[2]$, \dots, $T[k]$.
  The first and third condition on the function $h$ now imply
  \begin{align*}
    &h \pars*{
      \phi (\tilde{T}),
      \phi (T[1]) + \sum\nolimits_{j=2}^{k} \phi (T[j])
    }
    \\
    &= h \pars*{
      \phi (T [1] ),
      \phi (\tilde{T}) +
      \sum\nolimits_{j=2}^{k} \phi (T [j] )
    }
    \frac{
      \phi (\tilde{T}) \left( \sum_{j=1}^{k} \phi (T [j] ) \right)}{
      \phi (T[1]) \left(\phi (\tilde{T}) + \sum_{j=2}^{k} \phi (T [j] )\right)
    }.
  \end{align*}
  Write $\overline{T}$ for the tree obtained by replacing the subtree $T[1]$ by
  $\tilde{T}$ in $T$.
  The random tree $\overline{T}$ has the same law as $T$ and
  is independent of $T[1]$.
  Furthermore,
  \[
    \frac{
      \phi (\tilde{T})}{
      \phi (\tilde{T}) + \sum_{j=2}^{k} \phi (T[j])
    }
    =
    \frac{
      \phi (\overline{T}[1] )}{
      \sum_{j=1}^{k} \phi ( \overline{T} [j] )
    }.
  \]
  Plugging the last two equalities in equation~\eqref{eq:hh}, we obtain
  \[
    \E [ f (T [\Xi_1] ) \kappa ( \phi(T) ) ]
    =
    \sum_{k=1}^{\infty} kp_k
    \E \bracks*{
      f (T[1])
      \frac{
        \phi (\overline{T}[1])}{
        \sum_{j=1}^{k}\phi (\overline{T}[j])
      }
      h \pars*{
        \phi (T [1]),
        \sum\nolimits_{j=1}^{k} \phi (\overline{T}[j])
      }
    }.
  \]
  By symmetry, for all $i \leq k$, we have
  \begin{align*}
    &\E \bracks*{
      f (T[1])
      \frac{
        \phi (\overline{T}[1])}{
        \sum_{j=1}^{k}\phi (\overline{T}[j])
      }
      h \pars*{
        \phi (T[1]),
        \sum\nolimits_{j=1}^{k} \phi (\overline{T}[j])
      }
    }
    \\
    &=
    \E \bracks*{
      f (T[1])
      \frac{
        \phi (\overline{T}[i])}{
        \sum_{j=1}^{k}\phi (\overline{T}[j])
      }
      h \pars*{
        \phi (T[1]),
        \sum\nolimits_{j=1}^{k} \phi (\overline{T}[j])
      }
    },
  \end{align*}
  so that
  \begin{align*}
    &k\E \bracks*{
      f (T[1])
      \frac{\phi (\overline{T}[1])}{
        \sum_{j=1}^{k}\phi (\overline{T}[j])
      }
      h \pars*{
        \phi (T[1]),
        \sum\nolimits_{j=1}^{k} \phi (\overline{T}[j])
      }
    }
    \\
    &=
    \E \bracks*{
      f (T[1])
      h \pars*{
        \phi (T[1]),
        \sum\nolimits_{j=1}^{k} \phi (\overline{T}[j])
      }
    }.
  \end{align*}
  To conclude the proof, we remove the conditioning on the value of
  $\numch_T (\root) = \numch_{\overline{T}} (\root)$
  and replace
  $
    h (
      \phi (T [1]),
      \sum_{j=1}^{k} \phi (\overline{T} [j])
    )
  $
  by its conditional expectation given $T[1]$, which equals
  $\kappa (\phi (T[1]))$.
\end{proof}

\section{Dimension of the harmonic measure
  of the \texorpdfstring{$\lambda$}{lambda}-biased random walk}
\label{sec:lambda}
\subsection{The dimension as a function of the law of the conductance}
Let $\prt{t} \in \prt\noleaf$ and $\lambda > 0$.
The $\lambda$-biased random
walk on $t$ is the Markov chain whose transition probabilities are
the following :
\[
  P^t ( x , y ) =
  \begin{dcases}
    1 &  \text{if } x = \prt\root \text{ and } y = \root ;
    \\
    \frac{\lambda}{\lambda + \nu_t (x)} & \text{if } y = \prt x ;
    \\
    \frac{1}{\lambda + \nu_t (x)} & \text{if } y \text{ is a child of }x ;
    \\
    0 & \text{otherwise.}
  \end{dcases}
\]
For $x$ in $t$, we write $P^t_x$ for a probability measure under which the
process $ (X_n)_{n \geq 0} $ is the Markov chain on $t$, starting
from $x$, with transition kernel $P^t$.

Let $T$ be a $(\Gamma, \pp)$-Galton-Watson tree, where $\Gamma$ is
deterministic and will be fixed later on.
Recall that $m$ denotes the expectation of our reproduction law and that we
assume it to be finite.
From \cite{Lyons_rwperco}, we know that the $\lambda$-biased random walk on $T$
is almost
surely transient if and only if $\lambda < m$, which we assume from now on.
Since the walk is transient the random exit times defined by
\[
  \et_n
  \coloneqq
  \inf \setof{s \geq 0 }{\forall k \geq s, \, \abs{X_k} \geq n}
\]
are $P^T_\root$-almost surely finite.
We call
$ \Xi \coloneqq (X_{\et_n})_{n \geq 0}$
the harmonic ray et denote by $\harm_T$ its distribution.

We need to introduce the conductance to
state our result. For any tree $t$ in $\noleaf$,
\[
  \beta (t)
  \coloneqq
  P_\root^t (\forall k \geq 0, \, X_k \neq \prt{\root} ).
\]
This is the same as the effective conductance between $\prt\root$ and infinity in
the tree $t$ when we put on each edge $(\prt{x},x)$ the conductance
$\lambda^{-\abs{x}}$.
Notice that for any $x \in t$, we also have
\[
  \beta ( t[x] )
  = P_x^t (\forall k \geq 0, \, X_k \neq \prt{x} ).
\]
We denote, for $x \in \prt{t}$,
\[
  \tau_x \coloneqq \inf \setof{ k \geq 0 }{ X_k = x },
\]
with $\inf \emptyset = \infty$.
Then, applying successively the Markov property
at times $1$ and $\tau_\root$, we get
\begin{align*}
  \beta(t) & =P^t_\root (\tau_{\prt\root} = \infty)
  =
  \sum_{i=1}^{\numch_t(\root)}P^t(\root,i)
    P^t_i ( \tau_{\prt\root} = \infty)
  \\
  &=
  \sum_{i=1}^{\numch_t(\root)}
  P^t(\root,i)
  \left(
    P^t_i ( \tau_{\root} = \infty)
    + P^t_i ( \tau_\root < \infty)
    P^t_\root(\tau_{\prt\root} = \infty )
  \right)
  \\
  &=
  \sum_{i=1}^{\numch_t(\root)} P^t(\root,i)
  \left( \beta(t[i]) + (1-\beta(t[i])) \beta(t) \right).
\end{align*}
Some elementary algebra, together with the fact that
\[
  1 - \sum_{i=1}^{\nu_t(\root)}P^t(\root,i) = P^t(\root,\prt\root),
\]
lead to the recursive equation:
\begin{equation}\label{eq:recbetageneral}
  \beta (t) =
  \frac{
    \sum_{i = 1}^{\numch_t (\root)} P^t(\root, i) \beta (t[i]) }{
    P^t(\root,\prt\root) +
    \sum_{ i = 1 }^{\numch_t (\root) } P^t(\root, i) \beta (t[i])
  }.
\end{equation}
In our special case of the $\lambda$-biased random walk, this equation becomes:
\begin{equation}\label{eq:recbeta}
  \beta (t) =
  \frac{
    \sum_{i = 1}^{\nu_t (\root) } \beta (t[i])}{
    \lambda + \sum_{ i = 1 }^{\nu_t (\root) } \beta(t[i])
  }.
\end{equation}
For all $1 \leq i \leq \nu_t(\root)$, the harmonic measure of $i$ is:
\begin{align*}
  &P_\root (\Xi_1 = i)
  = P^t  (\root, \prt\root) P^t_{\prt\root}(\Xi_1 =1)
  + P^t (\root,i) P_i^t (\Xi_1 = i)
  + \sum_{\substack{j=1 \\ j\neq i}}^{\numch_t(\root)}
  P^t (\root, j) P_j^t (\Xi_1 = i) \\
  &=
  P_\root (\Xi_1 = i)
  \Big(
    P^t(\root, \prt\root) +
    \sum_{j=1}^{\nu_t(\root)} (1 - \beta(t[j])) P^t(\root, j)
  \Big)
    + P^t(\root,i)\beta(t[i]).
\end{align*}
As a consequence, for all $1 \leq i \leq \nu_t (\root) $,
\begin{equation}\label{eq:harmgeneral}
  \harm_t (i)
  = P_\root (\Xi_1 = i )
  = \frac{
    P^t(\root,i)\beta (t[i])}{
    \sum_{j = 1}^{\nu_t (\root) } P^t(\root,j)\beta (t[j])
  }.
\end{equation}
For the $\lambda$-biased random walk, this becomes,
\begin{equation}\label{eq:harmlambda}
  \harm_t (i)
  = \frac{
      \beta (t[i])}{
      \sum_{j = 1}^{\nu_t (\root) } \beta (t[j])
    }.
\end{equation}
This equation describes $\harm$ as a $\GW$-flow rule.
We are now ready to use the machinery described in \Cref{sec:invariant}.
Set $J = \intoo{0}{\infty}$ if $\lambda \geq 1$ and
$J = \intoo{1-\lambda}{\infty}$ if $\lambda < 1$.
In the latter case, the fact that almost surely
$\beta (T) > 1 - \lambda$
comes from the Rayleigh principle and the fact that $p_1 < 1$,
comparing the conductance of the whole tree to the one of a tree with a unique ray.
Set, for $u$ and $v$ in $J$,
\[
  h(u,v) = \frac{uv}{u+v+\lambda - 1}.
\]
Notice that we are in the setting of the second example of \Cref{sec:invariant}
if
$\lambda < 1$ and of the third if $\lambda \geq 1$, so that $h$ fulfills the
algebraic assumptions stated in Section~3.
By equation~\eqref{eq:recbeta},
\[
  \beta (t) =
  h \Bigl( 1, \sum\nolimits_{i = 1}^{\numch_t (\root) } \beta (t[i]) \Bigr).
\]
So we set $\Gamma_x \coloneqq 1$ for all $x \in T$.
For $u \in J$, let
\[
  \kappa (u)
  \coloneqq
  \E \bracks*{
    h \pars*{
      u,
      \sum\nolimits_{i = 1}^{\nu_{\tilde{T}}(\root)}
      \beta (\tilde{T}[i])
    }
  }
  =
  \E \bracks*{
    \frac{
      u \sum_{i = 1}^{\nu_{\tilde{T}}(\root)} \beta (\tilde{T}[i])}{
      \lambda - 1 + u +
      \sum_{i = 1}^{\nu_{\tilde{T}}(\root)} \beta (\tilde{T}[i])
    }
  },
\]
where $\tilde{T}$ is an independent copy of $T$.

We need to prove that
$ \E [ \kappa (\phi (T)) ] < \infty $.
By \cite[Lemma~{4.2}]{Elie_speed}:
\[
  \E \bracks*{\frac{1}{\lambda - 1 + \beta (T)}}< \infty.
\]
This implies, using the independence of $T$ and $\tilde{T}$, that
\begin{align*}
  \E \bracks*{
    \frac{
      \beta (T)
      \sum_{i = 1}^{\numch_{\tilde{T}}(\root)} \beta (\tilde{T}[i])}{
      \lambda - 1 + \beta (T) +
      \sum_{i = 1}^{\numch_{\tilde{T}}(\root)} \beta (\tilde{T}[i])
    }
  }
  &\leq
  \E \bracks*{
    \frac{
      \beta (T) \sum_{i = 1}^{\numch_{\tilde{T}}(\root)} \beta (\tilde{T}[i])}{
      \lambda - 1 + \beta (T)
    }
  }
  \\
  &= \E \bracks*{
    \sum\nolimits_{i = 1}^{\numch_{\tilde{T}}(\root)} \beta (\tilde{T}[i])
  }
  \E \bracks*{
    \frac{\beta (T)}{ \lambda - 1 + \beta (T)}
  }
  \\
  &\leq m \E \bracks*{ \frac{1}{\lambda - 1 + \beta (T)}} < \infty.
\end{align*}
We may now use \Cref{thm:alg}.
\begin{theorem}\label{thm:biased}
  The probability measure $\mu_\harm$ of density
  $C^{-1}\kappa (\beta (T))$,
  with respect to $\GW$,
  where $\kappa$ is defined by
  \[
    \kappa (u) =
    \E \bracks*{
      \frac{
        u \sum_{j=1}^{\numch_{\tilde{T}}(\root)} \beta(\tilde{T}[i])
      }{
        \lambda - 1 + u + \sum_{j=1}^{\numch_{\tilde{T}}(\root) }
        \beta(\tilde{T}[i])
      }
    },
  \]
  and $C$ is the renormalizing constant, is $\harm$-invariant.
The dimension of $\harm_T$ equals almost surely
\begin{equation}\label{eq:dlambda}
  d_\lambda = \log (\lambda)
    - C^{-1}
    \E \bracks*{
      \log (1 - \beta (T) )
      \frac{
        \beta(T)\sum_{i = 1}^{\numch_{\tilde{T}}(\root)} \beta (\tilde{T}[i])
      }{
        \lambda - 1 + \beta(T) +
        \sum_{i = 1}^{\numch_{\tilde{T}}(\root)} \beta (\tilde{T}[i])
      }
    }.
\end{equation}
\end{theorem}
\begin{proof}
  The only statement we still need to prove is the formula for the dimension.
  We write $\mu \coloneqq \mu_\harm$ for short and
  $
    \E_\mu [ \, \cdot \, ]
    \coloneqq
    \E \bracks{
      \, \cdot \, C^{-1}\kappa (\beta(T) )
    }
  $.
  By formula~\eqref{cor:hausdflow} and equation~\eqref{eq:harmlambda},
  \[
    d_\lambda =
    \E_\mu \bracks*{ \log \frac{1}{\harm_T (\Xi_1)} }
    = \E_\mu \bracks*{
      \log
      \frac{
        \sum_{i=1}^{\nu_T (\root) } \beta (T[i])}{
        \beta (T [\Xi_1] )
      }
    }.
  \]
  Using equation~\eqref{eq:recbeta}, we see that
  \[
    \sum\nolimits_{i=1}^{\nu_T (\root) } \beta (T[i])
    =
    \frac{\lambda \beta (T)}{1 - \beta (T)}.
  \]
  Therefore,
  \[
    d_\lambda =
    \log \lambda +
    \E_\mu \bracks*{
      - \log (1 - \beta (T)) + \log ( \beta (T) ) - \log ( \beta (T[\Xi_1]) )
    }.
  \]
  We are done if we can prove that
  $ \log(\beta (T)) - \log(\beta (T [\Xi_1])) $
  is integrable with integral $0$ with respect to $\E_\mu$.
  By invariance and \Cref{lem:classic},
  it is enough to show that it is bounded from below by an
  integrable function.
  We compute, using again formula~\eqref{eq:recbeta} :
  \[
    \frac{\beta (T)}{\beta ( T [\Xi_1] )}
    =
    \frac{
      1 + \beta ( T [\Xi_1] )^{-1}
      \sum_{i = 1, \, i\neq \Xi_1}^{\numch_T (\root)} \beta (T [i] )
      }{
      \lambda + \beta ( T [\Xi_1]  )
      + \sum_{i = 1, \, i\neq \Xi_1}^{\numch_T (\root)} \beta (T [i])
    }
    \geq \frac{1}{\lambda + \beta (T [\Xi_1])}
    \geq \frac{1}{\lambda + 1},
  \]
  where, for the first inequality, we used the fact that the function
  \[
    x \longmapsto
    \frac{
      \beta (T [\Xi_1] )^{-1} x + 1 }{
      x + \lambda + \beta (T [\Xi_1] )
    }
  \]
  is increasing on $\intfo{0}{\infty}$.
\end{proof}
To conclude this section, we recall that a very similar formula
was independently discovered by Lin
in~\cite{shen_lin_harmonic_biased},
using the same invariant probability measure.
Its work shows that such a formula, though it might look complicated, does
indeed yield very interesting results.

\subsection{Numerical results}
We will now use our formula to conduct numerical experiments about the dimension
$d_\lambda$ as a function of $\lambda$.

It was asked in \cite{LPP97} whether $d_\lambda$ is a monotonic function of
$\lambda$, for $\lambda$ in $\intoo{0}{m}$. To the best of our knowledge, this
question is still open. We were not able to find a theoretical answer.
However, using formula~\eqref{eq:dlambda} together with the recursive
equation~\eqref{eq:recbeta}, we are able to draw a credible enough graph
of $d_\lambda$ versus $\lambda$,
for a given (computationally reasonable) reproduction law.

The idea is the following.
Fix a reproduction law $\pp$
(such that $p_0=0$, $p_1 <1 $ and of finite mean $m$) and a bias $\lambda$ in
$\intoo{0}{m}$.
For any non-negative integer $n$, and a Galton-Watson tree $T$, let
$\beta_n(T) = P^T_\root \left(\tau^{(n)} < \tau_{\prt\root}\right)$, where
$\tau^{(n)}$ is the first hitting time of level $n$ by the $\lambda$-biased
random walk $ \left(X_n\right)_{n\geq 0}$.
Since the family of events $\{\tau^{(n)} < \tau_{\prt\root}\}$ is
decreasing, we have
\[
  \beta(T) = P_\root^T
  \Bigl(
    \bigcap_{n \geq 1} \left\{\tau^{(n)} < \tau_{\prt\root} \right\}
  \Bigr)
  = \lim_{n \to \infty} \beta_n(T).
\]
Using the Markov property as in~\eqref{eq:recbeta} yields the
recursive equation
\[
  \beta_{n+1} (T)
  = \frac{
      \sum_{i=1}^{\numch_T(\root)}\beta_n(T[i])}{
      \lambda + \sum_{i=1}^{\numch_T(\root)}\beta_n(T[i])
    }.
\]
By definition, $\beta_0(T)$ is equal to one. Hence, we may use the following
algorithm to compute the law of $\beta_n \coloneqq \beta_n(T)$:
\begin{itemize}
  \item the law of $\beta_0$ is the Dirac measure $\delta_1$;
  \item for any $n \geq 0$, assuming we know the law of $\beta_n$, the law of
    $\beta_{n+1}$ is the law of the random variable
    \[
      \frac{
        \sum_{i=1}^{\nu}\beta_n^{(i)}}{
        \lambda + \sum_{i=1}^{\nu}\beta_n^{(i)}
      },
    \]
    where $\nu$, $\beta_n^{(1)}$, $\beta_n^{(2)}$, ... are independent, $\nu$
    has the law $\pp$ and each $\beta_n^{(i)}$ has the law $\beta_n$.
\end{itemize}

Using the preceding algorithm, after $n$ iteration, we obtain the law of
$\beta_n(T)$. Since $\beta_n(T) \to \beta(T)$, almost surely,
we also have convergence in law.
\begin{remark}
  The preceding discussion shows that the law of $\beta$ is the greatest (for
  the stochastic partial order) solution of the recursive
  equation~\eqref{eq:recbeta}.
  In \cite[Theorem~4.1]{LPP97}, for $\lambda = 1$,
  the authors show that the only solutions to this recursive equation are the
  Dirac measure $\delta_0$ and the law of $\beta$.
  However, their proof cannot
  be adapted to the more general case $\lambda \in \intoo{0}{m}$.
  That is why, here, we had to choose for our initial measure, the Dirac measure
  $\delta_1$.
\end{remark}
For the numerical computations, we discretize the interval $\intff{0}{1}$ and
apply the preceding algorithm with some fixed final value of $n$.
See Figure~\ref{fig:beta123} for an example of what one can obtain with
$100$ iterations and a discretization step equal to $1 / 20000$.
\begin{figure}
  \centering
  \begin{tikzpicture}[scale=0.57]
    \begin{groupplot}[group style = {group size = 3 by 1}]
      \nextgroupplot[title = {$\lambda = 0.7$}, xmin=0.3, xmax=0.8]
        \addplot[mark=none, very thin, color = blue] table[x index=0, y index=1]
          {beta123lambda0d7.txt};
      \nextgroupplot[xmax=.7, title = {$\lambda = 1$}]
        \addplot[mark=none, very thin, color = blue] table[x index=0, y index=1]
          {beta123lambda1.txt};
      \nextgroupplot[xmax=.6, title = {$\lambda = 1.2$} ]
        \addplot[mark=none, very thin, color = blue] table[x index=0, y index=1]
          {beta123lambda1d2.txt};
    \end{groupplot}
  \end{tikzpicture}
  \caption{
    The apparent density of the conductance $\beta$, for
    $p_1 = p_2 = p_3 = \frac13$ and $\lambda$ in $\{0.7, 1, 1.2\}$.
  }
  \label{fig:beta123}
\end{figure}
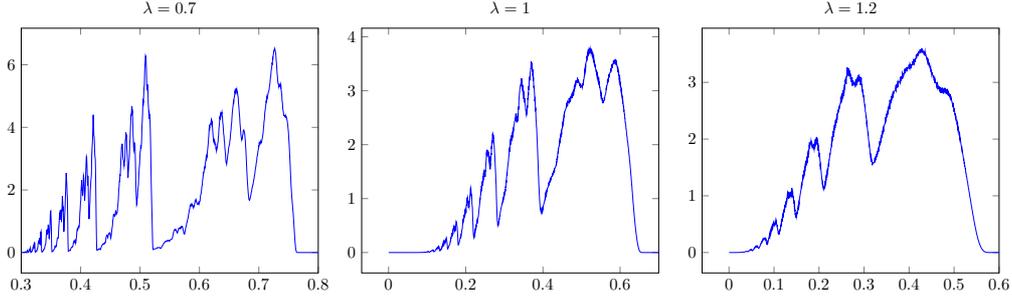

Before we compute the dimension, we simplify a little bit the
formula~\eqref{eq:dlambda}.
First, notice that we may write
\begin{equation}\label{eq:sumbetatoone}
  \sum_{j=1}^{\nu} \beta (\tilde{T}[j])
  = \frac{\lambda\tilde{\beta}}{1-\tilde{\beta}},
\end{equation}
where $\tilde{\beta}$ is an independent copy of $\beta = \beta(T)$.
Recalling that the constant $C$ in \eqref{eq:dlambda} is the expectation of
$\kappa(\beta)$, we obtain
\begin{equation}\label{eq:dlambda2}
  d_\lambda
  =
  \log(\lambda) -
  \left.
    \E \bracks*{
      \frac{
        \log (1 - \beta) \beta \tilde{\beta}}{
        \lambda - 1 + \beta + \tilde{\beta} - \beta\tilde{\beta}
      }
    }
  \middle/
    \E \bracks*{
        \frac{\beta \tilde{\beta}}{
        \lambda - 1 + \beta + \tilde{\beta} - \beta\tilde{\beta}
        }
      }
  \right..
\end{equation}
From there,
computing $d_\lambda$ from a dicretized approximation of the law of $\beta$ is
straightforward.

From \cite{shen_lin_harmonic_biased}, we know that $d_\lambda$ goes to
$\E\left[\log(\nu)\right]$
(the almost sure dimension of the visibility measure,
see \cite[Section~4]{LPP95}) as $\lambda$ goes to $0$,
and to $\log m$ as $\lambda$ goes to $m$.

We also compute numerically the speed. Recall from~\cite{Elie_speed}, that the
speed of the $\lambda$-biased transient random walk is given by
\[
  \ell_\lambda
  =
  \left.
    \E \bracks*{
      \frac{
        (\nu - \lambda)\beta_0}{
        \lambda - 1 + \beta_0 + \sum_{j=1}^{\nu} \beta_j
      }
    }
  \middle/
    \E \bracks*{
      \frac{
        (\nu + \lambda)\beta_0}{
        \lambda - 1 + \beta_0 + \sum_{j=1}^{\nu} \beta_j
      }
    }
  \right.,
\]
where $\nu$, $\beta_0$, $\beta_1$, \ldots are independent and $\nu$ has law
$\pp$, while for each $i$, $\beta_i$ has law $\beta(T)$.
Using first symmetry and then~\eqref{eq:sumbetatoone}, one obtains
\[
  \E \bracks*{
    \frac{
      (\nu \pm \lambda)\beta_0}{
      \lambda - 1 + \beta_0 + \sum_{j=1}^{\nu}\beta_i
    }
  }
  =
  \E \bracks*{
    \frac{
      \left(\sum_{j= 1}^{\nu}\beta_j\right) \pm \lambda\beta_0}{
      \lambda - 1 + \beta_0 + \sum_{j=1}^{\nu}\beta_j
    }
  }
  =
  \E \bracks*{
    \frac{
      \lambda \left(\tilde{\beta} \pm \beta \mp \beta\tilde{\beta}\right)}{
      \lambda -1 + \beta + \tilde{\beta} - \tilde{\beta}\beta
    }
  },
\]
where we have denoted $\beta = \beta_0$ and
$\sum_{j=1}^\nu \beta_i = \lambda\tilde{\beta}/(1 - \tilde{\beta})$.
Finally, we may express the speed as
\begin{equation}\label{eq:speed}
  \ell_\lambda
  =
  \left.
    \E \bracks*{
      \frac{
        \beta\tilde{\beta}}{
        \lambda - 1 + \beta + \tilde{\beta}
        - \beta\tilde{\beta}
      }
    }
  \middle/
  \E \bracks*{
      \frac{
        \beta + \tilde{\beta} - \beta\tilde{\beta}}{
        \lambda - 1 + \beta + \tilde{\beta} - \beta\tilde{\beta}
      }
    }
  \right..
\end{equation}
We also recall that, in the case $\lambda = 1$, it was shown in \cite{LPP95}
that the speed of the random walk
equals
\[
  \ell_1 = \E \bracks*{\frac{\nu -1}{\nu+1}}.
\]

We have made the numerical computations in two cases,
the first one is when the reproduction law is given by $p_1 = p_2 = 1/2$,
see \cref{fig:dimspeed12} and
the second one is for $\pp$ given by $p_1 = p_2=p_3 = 1/3$, see
\cref{fig:dimspeed123}.
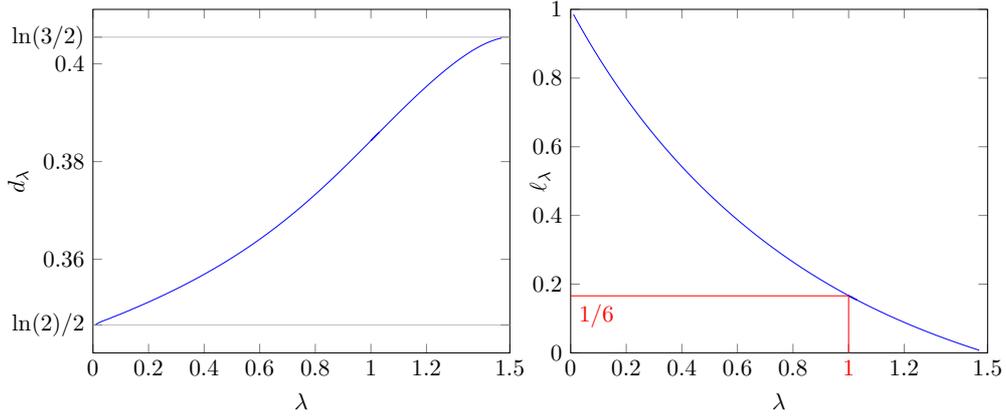
\begin{figure}
  \centering
  \begin{tikzpicture}[scale=0.8]
    \begin{groupplot}[group style = {group size = 2 by 1}]
    \nextgroupplot[
      extra y ticks={0.34657, 0.40547},
      extra y tick labels={$\ln(2)/2$, $\ln(3/2)$},
      extra y tick style={grid=major},
      xlabel=$\lambda$,
      ylabel=$d_\lambda$,
      xtick = {0, .2, .4, .6, .8, 1., 1.2, 1.5},
      xmin = 0,
      xmax = 1.5
      ]
      \addplot[mark=none, very thin, color = blue] table[x=lambda, y=dim]
        {sd12.txt};
    \nextgroupplot[
      xmin = 0,
      xmax = 1.5,
      ymin = 0,
      ymax = 1.,
      xlabel=$\lambda$,
      ylabel=$\ell_\lambda$,
      ytick = {0, 0.2, 0.4, 0.6, 0.8, 1.},
      y label style={at={(axis description cs:0.15,.5)},anchor=south},
      xtick = {0, .2, .4, .6, .8, 1.2, 1.5},
      extra x ticks = {1},
      extra x tick style = {tick label style = red},
      enlarge x limits = false,
      enlarge y limits = false,
      ]
      \addplot[mark=none, very thin, color = blue] table[x=lambda, y=speed]
        {sd12.txt};
      \draw[help lines, color = red] (axis cs:0,0.166) -| (axis cs:1,0);
      \node[below right, color = red] at (axis cs:0,0.167) {$1/6$};
    \end{groupplot}
  \end{tikzpicture}
  \caption{
    The dimension and the speed of the $\lambda$-biased random walk on a Galton-Watson
    tree as functions of $\lambda$, for $p_1 = p_2 = 1/2$.
  }
  \label{fig:dimspeed12}
\end{figure}

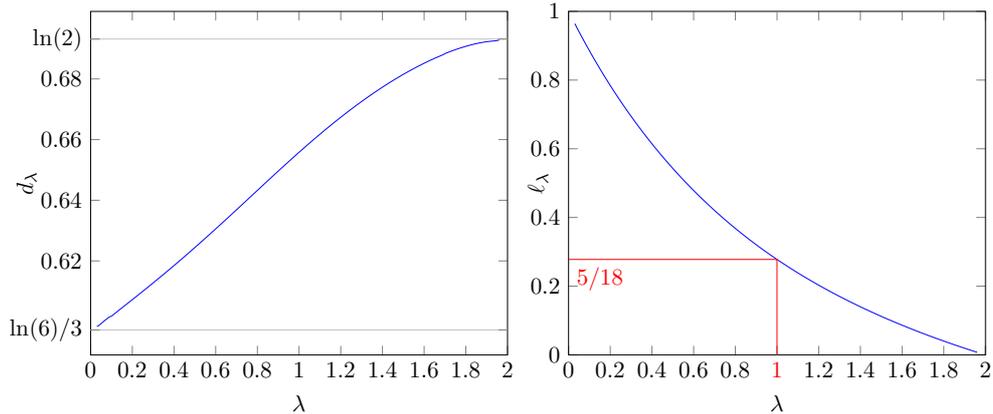
\begin{figure}
  \centering
  \begin{tikzpicture}[scale=0.8]
  \begin{groupplot}[group style = {group size = 2 by 1}]
    \nextgroupplot[
      extra y ticks={0.59725, 0.69315},
      extra y tick labels={$\ln(6)/3$, $\ln(2)$},
      extra y tick style={grid=major},
      xlabel=$\lambda$,
      ylabel=$d_\lambda$,
      y label style={at={(axis description cs:0.07,.5)},anchor=south},
      ytick = {0.62, 0.64, 0.66, 0.68},
      xtick = {0, .2, .4, .6, .8, 1., 1.2, 1.4, 1.6, 1.8, 2},
      xmin = 0,
      xmax = 2.0
      ]
      \addplot[mark=none, very thin, color = blue] table[x=lambda, y=dim]
        {sd123.txt};
    \nextgroupplot[
      xmin = 0,
      xmax = 2,
      ymin = 0,
      ymax = 1,
      xlabel=$\lambda$,
      ylabel=$\ell_\lambda$,
      y label style={at={(axis description cs:0.15,.5)},anchor=south},
      xtick = {0, .2, .4, .6, .8, 1.2, 1.4, 1.6, 1.8, 2},
      extra x ticks = {1},
      extra x tick style = {tick label style = red},
      ytick = {0, 0.2, 0.4, 0.6, 0.8, 1.},
      enlarge x limits = false,
      enlarge y limits = false,
      ]
      \addplot[mark=none, very thin, color = blue] table[x=lambda, y=speed]
        {sd123.txt};
      \draw[help lines, color = red] (axis cs:0,0.2778) -| (axis cs:1,0);
      \node[below right, color = red] at (axis cs:0,0.2778) {$5/18$};
  \end{groupplot}
  \end{tikzpicture}
  \caption{
    The dimension and the speed of the $\lambda$-biased random walk on a
    Galton-Watson tree as functions of $\lambda$, for $p_1 = p_2 = p_3 = 1/3$.
  }
  \label{fig:dimspeed123}
\end{figure}
These figures suggest that the
speed and the dimension are indeed monotonic with respect to $\lambda$.
Furthermore, the speed looks convex, while the dimension seems to be neither
convex nor concave.

\section{Galton-Watson trees with recursive lengths}
\label{sec:reclength}
\subsection{Description of the model}
We generalize a model of trees with random lengths
(or resistances) that can be found in
\cite[Section~2]{Curien_LeGall_harmonic} and
\cite[Section~2]{Shen_harmonic_infinite_variance}.
It appeared as the scaling limit of the sequence
$( T_n / n)_{n\geq 1}$,
where $T_n$ is a reduced critical Galton-Watson
tree conditioned to survive at the $n$\textsuperscript{th} generation.

In both \cite{Curien_LeGall_harmonic} and
\cite{Shen_harmonic_infinite_variance},
the marks have the law of the inverse of a uniform random variable on
$\intoo{0}{1}$.
The reproduction law is $p_2 = 1$ in \cite{Curien_LeGall_harmonic},
whereas in \cite{Shen_harmonic_infinite_variance} it is given by
\begin{equation}\label{eq:reprod_lin}
  p_k =
  \begin{dcases}
    0 & \text{if } k \leq 1; \\
    \frac{
      \alpha \Gammafunc \left(k- \alpha\right)}{
      k!  \Gammafunc(2-\alpha)
    }
    & \text{otherwise,}
  \end{dcases}
\end{equation}
where $\alpha$ is a parameter in $ \intoo{1}{2}$, and $\mathsf\Gamma$ is the standard
Gamma function.

Here, we assume that the marks are in $\intoo{1}{\infty}$, and as before, that
$p_0 = 0$ and $p_1 < 1$, but some of our results will need some integrability
assumptions.

Let $t$ be a marked leafless tree with marks in $\intoo{1}{\infty}$.
We associate to each vertex $x$ in $t$, the resistance, or length,
of the edge $(x_*,x)$:
\[
  r_t (x) =
    \gamma_t(x)^{-1}\prod_{y \prec x }^{}
    \left(1 - \gamma_t(y)^{-1}\right) .
\]
Informally, the edge between the root and its parent has length
$\gamma_t(\root)^{-1} \in \intoo{0}{1}$.
Then we multiply all the lengths in the subtrees
$t[1]$, $t[2]$, \dots, $t[\nu_t(\root)]$ by
$\left( 1 - \gamma_t(\root)^{-1}\right)$ and we
continue recursively see \Cref{fig:recursive_lengths}.
We run a nearest-neighbour random walk on the tree with transition
probabilities inversely proportional to the lengths of the edges (further
neighbours are less likely to be visited).
To make this more precise,
the random walk in $\prt{t}$,
associated to this set of resistances has the following
transition probabilities:
\[
  Q^t (x, y) =
  \begin{cases}
    1
    & \text{if } x = \prt\root \text{ and } y = \root ;
    \\
    \gamma_t(xi)/
    \left(\gamma_t(x) - 1 + \sum_{j=1}^{\numch_t (x)} \gamma_t(xj)\right)
    & \text{if } y = xi \text{ for some } i \leq \numch_t(x)
    \\
    \left(\gamma_t(x) - 1\right) /
    \left(\gamma_t(x) - 1 + \sum_{j=1}^{\numch_t (x)} \gamma_t(xj) \right)
    & \text{if } y = \prt{x} ;
    \\
    0 & \text{otherwise.}
  \end{cases}
\]
When we reindex a subtree, we also change the resistances to gain stationarity.
For $x \in t$ and $y \in t[x]$, we define
\[
  r_{t[x]} (y) \coloneqq \frac{r_t (xy)}{
  \prod_{z \prec x}^{} \left(1 - \gamma_t(z)\right)^{-1}}.
\]
This is consistent with the marks of the reindexed subtree $t[x]$ and
does not change the probability transitions of the random walk inside
this subtree.
For $x$ in $\prt{t}$, let $Q^t_x$ be a probability measure under which the process
$\left(Y_n\right)_{n \geq 0}$ is the random walk on $t$ starting from $x$ with
probability transitions given by $Q^t$.
To prove that this walk is almost surely transient, we use Rayleigh's principle
and compare the resistance of the whole tree between $\prt\root$
and infinity to the resistance of, say, the
left-most ray. If, for $n$ greater or equal to one,
we denote by $r_n(t)$ the resistance in the ray between
$\prt\root$ and the vertex
$x_n \coloneqq \underbrace{11\cdots 1}_{n \text{ times}}$
, we have that
\begin{equation}\label{eq:resoneray}
  1 - r_n(t) = \left(1 - \gamma_t(\root)^{-1}\right)
  \left(1 - \gamma_t(1)^{-1}\right)
  \dotsm
  \left(1 - \gamma_t(x_n)^{-1}\right),
\end{equation}
so the resistance of the whole ray is less or equal to $1$.
In particular, it is
finite and so is the resistance of the whole tree.
We denote by $\harm^\Gamma_t$ the law of the exit ray of this random walk. For
$x$ in $\prt{t}$, let
\[
  \tau_x = \inf \setof{ k \geq 0}{ Y_k = x},
\]
with $\inf \emptyset = \infty$ and
\[
  \beta(t ) \coloneqq
  Q_\root^t (\tau_{\prt{\root}} = \infty ).
\]
Applying equations \eqref{eq:recbetageneral} and \eqref{eq:harmgeneral} to this
model, we obtain:
\begin{equation}\label{eq:harmgamma}
  \harm^\Gamma_t (i) =
  \frac{
    \gamma_t(i) \beta (t [i] )}{
    \sum_{i=1}^{\numch_t(\root)} \gamma_t(j) \beta (t [j] )
  }, \quad
  \forall i \leq \numch_t(\root),
\end{equation}
\begin{equation}\label{eq:recbetagamma}
  \gamma_t(\root) \beta (t)
  =
  \frac{
    \gamma_t(\root)
    \sum_{j=1}^{\numch_t(\root)} \gamma_t(j) \beta ( t[j] )}{
    \gamma_t(\root) - 1 + \sum_{j=1}^{\numch_t(\root)} \gamma_t(j)
    \beta ( t[j] )
  }.
\end{equation}
Let
\begin{equation}\label{eq:defphi}
  \phi (t) \coloneqq \gamma_t(\root) \beta(t)
  = \gamma_t (\root) Q^t_\root( \tau_{\prt\root} = \infty ).
\end{equation}
In fact,
$\phi \left(t\right)$ is the effective conductance between $\prt\root$ and
infinity.
From the identity~\eqref{eq:resoneray}, the Rayleigh principle
and the law of parallel conductances,
whenever $t$ has at least two rays,
$\phi(t) > 1$.
Thus we can write
\begin{equation}\label{eq:recphi}
  \phi (t) =
  h \Bigl(
    \gamma_t(\root),
    \sum_{i=1}^{\nu_t (\root)} \phi (t[i])
  \Bigr),
\end{equation}
with $h (u, v ) \coloneqq uv / (u + v - 1)$ for all $u$ and $v$ in
$J \coloneqq \intoo{1}{\infty}$.

Now, let $T$ be a $(\Gamma, \pp)$-Galton-Watson tree, where $\Gamma$ almost surely
belongs to $\intoo{1}{\infty}$, $p_0 = 0$ and $p_1 < 1$, so that $T$ almost
surely has infinitely many rays.

\subsection{Invariant measure and dimension drop for the natural distance}

We set for all $u > 1$,
\begin{equation}\label{eq:defkappa}
  \kappa (u)
  \coloneqq
  \E \Biggl[
    h \Bigl(
      u,
      \sum_{j=1}^{\numch_{\tilde{T}}(\root)} \phi (\tilde{T} [j])
    \Bigr)
  \Biggr]
  = \E \bracks*{
    \frac{
      u \sum_{j=1}^{\numch_{\tilde{T}}(\root)} \phi (\tilde{T} [j]) }{
      u - 1 +
      \sum_{j=1}^{\numch_{\tilde{T}}(\root)} \phi (\tilde{T} [j])
    }
  }.
\end{equation}
where $\tilde{T}$ is an independent copy of $T$.
We will be able to use \Cref{thm:alg} if we can prove
that $\kappa \left(\phi (T) \right)$ is integrable.
To this end, one needs some information about the law of $\Gamma$
and about $\pp$.
The following criterion is certainly not sharp but it might suffice in some
practical cases.
For its proof, we rely on ideas from
\cite[Proposition~6]{Curien_LeGall_harmonic}.
\begin{proposition}
  Assume that there exist two positive numbers $a$ and $C$ such that
  for all numbers $s$ in $\intoo{1}{\infty}$,
  $\P (\Gamma \geq s) \leq C s^{-a}$. Then,
  $ \E [\phi(T)] $ and $\E [\kappa (\phi(T))]$
  are finite whenever one of the following conditions occurs:
  \begin{enumerate}
    \item $a > 1$;
    \item $a = 1$ and $\sum_{k \geq 1} p_k k \log k < \infty$;
    \item $0 < a < 1 $ and $\sum_{k \geq 1} p_k k^{2 - a} < \infty$.
  \end{enumerate}
\end{proposition}
\begin{proof}
  From the fact that for all real numbers $u$ and $v$ greater that $1$,
  $h(u,v) < u$, we deduce that
  $\E[\kappa (\phi(T))]$ is finite
  as soon as $\E[\phi(T)]$ is, and we also conclude in the first case.

  Let $\mathcal{M}$ be the set of all Borel probability measures on
  $\intof{1}{\infty}$.
  For any distribution $\mu$ in $\mathcal{M}$, let
  $\Psi (\mu)$ be the distribution of
  $h (\Gamma, \sum_{i=1}^{\nu}X_i)$,
  where $\numch$, $\Gamma$ and $X_1$, $X_2$, \ldots are independent,
  each $X_i$ having
  distribution $\mu$ and $\numch$ having distribution $\pp$.
  To handle the case where
  $\mu(\{\infty\}) > 0$, we define by continuity
  $h(u, \infty) = u$ for all $u > 1$.
  Consider for any $s \in \intoo{1}{\infty}$,
  $
    F_\mu (s) \coloneqq \mu \intff{s}{\infty},
  $
  with $F_\mu( s) = 1$ if $s \leq 1$.
  On $\mathcal{M}$, the stochastic partial order $\preceq$ is defined as
  follows: $\mu \preceq \mu'$ if and only if there exists a coupling
  $(X, X')$ in some probability space, with $X$ distributed as $\mu$,
  $X'$ distributed as $\mu'$ such that $X \leq X'$ almost surely.
  This is equivalent to $F_\mu \leq F_{\mu'}$.
  From the identity
  \begin{equation}
    h(u,v) - h(u,v') = (v - v')
    \frac{
      u(u-1)}{
      (u+v-1)(u+v'-1)
    },
  \end{equation}
  we see that $\Psi$ is increasing with respect to the stochastic partial order.

  Let us denote by $\varphi$ the distribution of $\phi(T)$ and
  by $\gamma$ the distribution of $\Gamma$.
  Since $\Psi(\delta_\infty) = \gamma$ and $\Psi(\varphi) = \varphi$, we have
  $\varphi \preceq \Psi^n (\gamma)$ for all $n \geq 1$.
  We are done if we can show that
  $\Psi^n (\gamma)$
  has a finite first moment for some $n \geq 1$.

  Let $\mu$ be in $\mathcal{M}$ and $s \in \intoo{1}{\infty}$,
  \begin{align}
    F_{\Psi(\mu)} (s)
      &=
      \P \pars*{
        \Gamma \geq s , \, \sum_{i=1}^{\nu} X_i \geq  s \left(
        \frac{\Gamma - 1}{\Gamma - s}\right)
      }
    \nonumber \\
      &\leq
      \P \pars*{ \Gamma \geq s, \, \sum_{i=1}^{\nu}X_i \geq s }
    \nonumber \\
      &= \P (\Gamma \geq s )
      \P \pars*{ \sum_{i=1}^{\nu} X_i \geq s }
    \nonumber \\
      &= F_\gamma (s) \sum_{k \geq 1}^{} p_k
      \P \pars*{ \sum_{i= 1}^k X_i \geq s }
    \nonumber \\
      &\leq F_\gamma (s) \sum_{k \geq 1}^{} k p_k
      F_{\mu} \pars*{\frac{s}{k}}.
    \label{eq:Fineq}
  \end{align}

  We may apply it to $\gamma$, to get
  \begin{align*}
    \int_1^\infty F_{\Psi(\gamma)}(s) \dd s
      &\leq
      \sum_{k \geq 1}^{} k p_k
      \left(
        \int_1^k F_\gamma (s) \dd s
        + \int_k^\infty F_\gamma(s) F_\gamma \pars*{ \frac{s}{k} } \dd s
      \right)
    \\
      &=
      \sum_{k \geq 1}^{} k p_k
      \left(
        \int_1^k F_\gamma(s) \dd s
        + k\int_1^\infty F_\gamma (s) F_\gamma (k s)\dd s
      \right).
  \end{align*}
  In the second case, where
  $F_\gamma (s) \leq C s^{-1}$ and
  $ \sum_{k \geq 1}^{} p_k k \log k < \infty$,
  this is enough to conclude.

  In the third case, we need to play this game a little bit longer.
  Let $N \geq 1$ be the smallest integer such that $a \left(N + 1\right) > 1$.
  Notice that this implies that $aN \leq 1$.
  Iterating on the inequality~\eqref{eq:Fineq}
  and applying it to $\gamma$, we get
  \begin{align*}
    &F_{\Psi^N(\gamma)} (s)
    \\
    &\leq
    \sum_{ k_1, k_2, \dotsc, k_N \geq 1}^{}
    k_1 k_2 \cdots k_N p_{k_1} p_{k_2} \cdots p_{k_N}
    F_\gamma \Bigl(s \Bigr)
    F_\gamma \Bigl( \frac{s}{k_1} \Bigr)
    F_\gamma \Bigl( \frac{s}{k_1 k_2} \Bigr)
    \cdots
    F_\gamma \Bigl( \frac{s}{k_1 k_2 \cdots k_N }\Bigr).
  \end{align*}
  Hence, we may write an upper bound of
  $\int_1^\infty F_{\Psi^N(\gamma)}(s)\dd s$
  as
  \begin{gather*}
    \begin{split}
    \sum_{k_1, \dotsc, k_N \geq 1}^{}
      k_1 \dotsm k_N
      p_{k_1} \dotsm p_{k_N}
      \Bigl[
        I_1 (k_1) + I_2 (k_1, k_2) &+ \dotsb + I_N (k_1, \dotsc, k_N)
      \\
       &+ J (k_1, \dotsc, k_N)
      \Bigr],
    \end{split}
    \shortintertext{where}
    I_1 (k_1) \coloneqq \int_1^{k_1} F_\gamma (s) \dd s
      \leq \frac{C}{1-a} k_1^{1 - a}; \\
    \shortintertext{for $r$ between $2$ and $N$,}
    \begin{aligned}
    I_r (k_1, \dotsc, k_r)
      &\coloneqq
        \int_{k_1 \dotsm k_{r-1}}^{k_1 \dotsc k_r}
          F_\gamma(s) F_\gamma( s/k_1) \dotsm
          F_\gamma \bigl(s / (k_1 \dotsm k_{r-1})\bigr) \dd s
    \\
      &= k_1 \dotsm k_{r-1} \int_{1}^{k_r}
        F_\gamma (s) F_\gamma (sk_{r-1})
        \dotsm F_\gamma (s k_{r-1} \dotsm k_1 ) \dd s
    \\
      &\leq
      \begin{cases}
        k_1 \dotsm k_{r-1} C^{r} \log(k_r)
        k_{r-1}^{-a(r-1)} \dotsm k_1^{-a}
        & \text{if $r = N$ and $aN = 1$;}
        \\
        k_1 \dotsm k_{r-1} \frac{1}{1-ar}{C^{r}} k_r^{1 -ar}
        k_{r-1}^{-a(r-1)} \dotsm k_1^{-a}
        & \text{otherwise;}
      \end{cases}
    \\
      &\leq \tilde{C} k_1^{1 -a} \dotsc k_r^{1 - a},
    \end{aligned}
  \shortintertext{where $\tilde{C}$ is the positive constant defined by}
  \tilde{C} =
  \begin{cases}
    \max_{2\leq r \leq N} \left(C^r / (1  - ar) \right)
    & \text{if $aN < 1$};
    \\
    \max
    \left(
      \max_{2\leq r \leq N-1}\left( C^r / (1  - ar)\right),
      C^N \sup_{k\geq 1} \left(k^{a-1} \log (k)\right)
    \right)
    & \text{if $aN=1$.}
  \end{cases}
  \shortintertext{Finally,}
  \begin{aligned}
    J(k_1,\dotsc,k_N)
      &\coloneqq \int_{k_1\dotsm k_N}^{\infty}
      F_\gamma (s) F_\gamma \left(s / k_1 \right) \dotsm
      F_\gamma (s / (k_1 \dotsm k_N) ) \dd s
    \\
      &\leq C^{N+1} k_1^{1-a} \dotsm k_N^{1-a} \int_1^\infty s^{-a(N+1)} \dd s
    \\
      &=
      \frac{C^{N+1}}{a(N+1) - 1} k_1^{1-a} \dotsm k_N^{1-a},
      \quad
      \text{by our assumption that $a(N+1) > 1$.}
  \end{aligned}
  \end{gather*}
  The condition
  $\sum_{k\geq 1} p_k k^{2-a} < \infty$
  ensures that all the above sums are finite.
\end{proof}
\begin{example}
  If the law of $\Gamma^{-1}$ is uniform on $\intoo{0}{1}$ (as in
  \cite{Curien_LeGall_harmonic} and \cite{Shen_harmonic_infinite_variance}),
  we have, for
  any $s$ in $\intoo{1}{\infty}$, $\P (\Gamma \geq s) = s^{-1}$ and the previous
  proposition shows that
  $\E [ \kappa(\phi(T))]$
  is finite if
  $\sum^{}_{k=1} p_k k \log k < \infty$.
  If the reproduction law is the same as in
  \cite{Shen_harmonic_infinite_variance}, that is, is
  given by \eqref{eq:reprod_lin}, then by a well-known equivalent on gamma
  function ratios (see for instance~\cite{jameson_gamma}), we have
  \[
    p_k =
    \frac{\alpha}{\Gammafunc(2 - \alpha)}
    \frac{\Gammafunc(k - \alpha)}{\Gammafunc(k+1)}
    \sim_{k \to \infty}
    \frac{\alpha}{\Gammafunc(2 - \alpha)}k^{-1-\alpha},
  \]
  with $\alpha$ in $\intoo{1}{2}$.
  Thus $\sum_{k\geq 1} p_k k \log(k)$ is finite
  and so is
  $\E[ \kappa(\phi(T))]$.
\end{example}
With some more knowledge of $\pp$ and/or the law of $\Gamma$,
it could be possible to describe more precisely the law of $\phi(T)$.
See for instance \cite[Proposition~5]{Shen_harmonic_infinite_variance} or
\cite[Proposition~6]{Curien_LeGall_harmonic}.
However, in general, it is often very difficult to establish properties (for
instance, absolute continuity) of probability measures defined by distributional
recursive equations like \eqref{eq:recphi}.

We now apply \cref{thm:alg} to our problem and prove that the dimension drop
phenomenon occurs when the metric is the natural distance $\distrays$, defined
by \eqref{eq:defdistnat}.

\begin{theorem}\label{thm:reclengthd}
  Let $T$ be a $(\Gamma, \pp)$-Galton-Watson tree.
  Let $\phi(T)$ and $\kappa$ be defined respectively by~\eqref{eq:defphi}
  and~\eqref{eq:defkappa}.
  Assume that
  $C \coloneqq \E [\kappa \left(\phi(T)\right)]$
  is finite.
  Then, the probability measure of density
  $C^{-1} \kappa (\phi (T))$
  with respect to $\GW$ is invariant and ergodic with respect to the flow rule
  $\harm^\Gamma$.

  The dimension of the measure $\harm_T^\Gamma$ on $\rays{T}$ with
  respect to $\distrays$ equals
  almost surely
  \begin{equation}\label{eq:dimgamma}
    C^{-1}
    \E \bracks*{
      \log \pars*{
        \frac{1 - \Gamma_\root^{-1}}{1 - \Gamma_\root^{-1}\phi(T)}
      }
      \kappa (\phi(T))
    }.
\end{equation}
  It is almost surely strictly less than $\log m$ unless both $\pp$ and the
  law of $\Gamma$ are degenerated.
\end{theorem}
\begin{proof}
  The first part of the theorem is a direct consequence of
  \Cref{thm:alg}.

  We prove the formula for the dimension in the same way as in the
  previous theorem and use the same notations.
  Write $\mu$ for the probability measure with density
  $C^{-1}\kappa(\phi(T))$ with respect to $\GW$.
  Then by formula~\eqref{cor:hausdflow}, invariance of $\mu$ and
  equality~\eqref{eq:harmgamma} the dimension of
  $\harm_T^\Gamma$ equals almost surely
  \[
    \dim^{\distrays} \harm_T^\Gamma
    =
    \E_\mu \bracks*{
      \log \frac{1}{\harm_T^\Gamma (\Xi_1)}
    }
    =
    \E_\mu \bracks*{
      \log
      \frac{
        \sum_{i=1}^{\nu_T (\root)}\phi (T [i])}{
        \phi (T [\Xi_1])
      }
    }.
  \]
  From formula~\eqref{eq:recbetagamma}, we deduce that
  \[
     \sum_{i=1}^{\numch_T (\root)}\phi (T [i])
     =
     \frac{
       \phi (T) (1 - \Gamma_\root^{-1})}{
        1 - \Gamma_\root^{-1}\phi(T)
      },
  \]
  so that almost surely,
  \[
    \dim \harm_T^\Gamma =
    \E_\mu \bracks*{
      \log (1 - \Gamma_\root^{-1} )
      - \log (1 - \Gamma_\root^{-1}\phi (T))
      + \log (\phi \left(T\right))
      - \log (\phi (T [\Xi_1] ))
    }.
  \]
  As before, we need to prove that
  $
    \log (\phi (T)) - \log (\phi (T [\Xi_1] ))
  $
  is bounded from
  below by an integrable random variable to conclude. Using again
  formula~\eqref{eq:recbetagamma}, we get
  \begin{align*}
    \frac{\phi (T)}{\phi (T [\Xi_1] )}
    &=
    \frac{
      1 + \phi (T [\Xi_1] )^{-1}
      \sum_{i= 1, \, i\neq \Xi_1}^{\numch_T(\root)} \phi (T [i])
    }{
      1 - \Gamma_\root^{-1} +
      \Gamma_\root^{-1}\phi (T [\Xi_1] )
      +
      \Gamma_\root^{-1}
      \sum_{i= 1, \, i\neq \Xi_1}^{\numch_T(\root)} \phi (T [i])
    }
    \\
    &\geq
    \frac{1}{
      1 - \Gamma_\root^{-1} + \Gamma_\root^{-1} \phi (T [\Xi_1] )
    }.
  \end{align*}
  Hence, since $1 - \Gamma_\root^{-1} \leq 1$ and
  $
    \Gamma_\root^{-1} \phi (T [\Xi_1] ) =
    \beta ( T [\Xi_1] ) \leq 1
  $,
  we have
  \[
    \frac{
      \phi (T)}{
      \phi (T [\Xi_1] )
    }
    \geq \frac12.
  \]

  To prove the dimension drop, i.e.~the fact that almost surely
  $\dim \harm^\Gamma < \log m$, we do not use the formula~\eqref{eq:dimgamma}
  since we know so little about the distribution of $\phi(T)$.
  Instead, we compare the flow rule $\harm^\Gamma$ to the uniform flow
  $\unif$ defined in \Cref{sec:flowrules}.

  By \Cref{prop:dimensiondrop} and \Cref{lem:equalflows}
  we only need to prove that if
  there exists a positive real number $K$ such
  that, for $\GW$-almost every tree $t$,
  $\tilde{W} (t) = K \times \phi (t)$, then
  both the reproduction law and the mark law are degenerated.

  Under this assumption, by the recursive equation~\eqref{eq:recunif}
  we have almost
  surely,
  \[
    \phi \left(T \right)
    =
    \frac{1}{m}
    \sum_{i=1}^{\nu_T(\root)} \phi (T[i]).
  \]
  Plugging it into the recursive equation~\eqref{eq:recbetagamma},
  we first obtain that
  \[
    \phi(T) =
    \frac{
      m \Gamma_\root \phi(T)}{
      \Gamma_\root + m \phi(T) - 1
    },
  \]
  so that almost surely
  \[
    \phi (T) = \frac{1}{m} \left( (m-1) \Gamma_\root + 1\right).
  \]
  In turn, using again \eqref{eq:recbetagamma}, this implies that
  \[
    \frac{1}{m} \left[ \left(m-1\right)\Gamma_\root + 1 \right]
    =
    \frac{
      \Gamma_\root \sum_{i=1}^{\nu_T(\root)}
      \frac{1}{m} \left[(m-1)\Gamma_i + 1 \right]
    }{
      \Gamma_\root - 1 + \sum_{i=1}^{\nu_T(\root)} \frac{1}{m}
      \left[(m-1) \Gamma_i + 1 \right]
    }.
  \]
  Now, if we denote by $S$ the random variable $\sum_{i=1}^{\nu_T(\root)} \frac{1}{m}
    \left[(m-1) \Gamma_i + 1 \right]$, elementary algebra leads to the second
    degree polynomial equation
  \[
    (m-1) \Gamma_\root^2 + \Gamma_\root (2 - m - S) + S - 1 = 0,
  \]
  whose discriminant is equal to $ \left(S - m\right)^2 $.
  Hence, we always have
  \[
    \Gamma_\root =
    \frac{
      m + S - 2 \pm (S - m)}{
      2(m-1)
    }.
  \]
  We must choose the solution
  $\Gamma_\root = (S - 1) / (m - 1)$,
  because the other solution is $1$, which we forbid.
  As a consequence,
  \[
    \Gamma_\root =
    \frac{1}{m}
    \sum_{i=1}^{\numch_T (\root)}\Gamma_i +
    \frac{1}{m - 1} \left( \frac{\numch_T(\root)}{m} - 1 \right).
  \]
  which, by independence of $\numch_T(\root)$, $\Gamma_\root$, $\Gamma_1$,
  $\Gamma_2$, \ldots, imposes that both $\pp$ and the law of $\Gamma$ are
  degenerated.
\end{proof}

\subsection{Dimension and dimension drop for the length metric}
\label{subsec:dimdroplength}
Note that in the previous theorem, the dimension is computed with
respect to the natural distance $\distrays$.
This distance does not take into account the marks
$ \left(\Gamma_x\right)_{x \in T}$,
so we do not compute the same dimension as in
\cite{Curien_LeGall_harmonic} and \cite{Shen_harmonic_infinite_variance},
where the distance between two points in the tree is the sum of all the
resistances (or lengths) of the edges between these two points.

To make this definition more precise,
let us define, for $x \in T$, the $\Gamma$-height of $x$:
\[
  \abs{x}^\Gamma \coloneqq \sum_{y \preceq x}^{}
  \Big(
    \prod_{z \prec y}^{} \left(1 - \Gamma_z^{-1}\right)
  \Big)
  \Gamma_y^{-1}.
\]
We then have
\[
  1 - \abs{x}^\Gamma
  = \prod_{y \preceq x} \left(1 - \Gamma_y^{-1}\right).
\]
For two distinct rays $\eta$ and $\xi$, let
\[
  \distgamma(\xi, \eta ) \coloneqq 1 - \abs{\xi\wedge\eta}^\Gamma.
\]
Notice that, for any rays $\xi$ and $\eta$, and all integer $n \geq 1$,
we have:
\begin{equation}\label{eq:ballsgamma}
  \distgamma (\xi, \eta ) \leq 1 - \abs{\xi_n}^\Gamma
  \iff
  \eta \in \cyl[T]{\xi_n},
\end{equation}
where we recall that $\cyl[t]{\xi_n}$ is the set of all rays whose $\xi_n$
is a prefix.

We will compute the dimension of $\harm^\Gamma_T$ with respect to
this distance $\distgamma$ and show that in this case too, we observe a
dimension drop phenomenon, but we begin with more general statements.
We want to build a theory similar to \cite[Sections~6 and 7]{LPP95} in
our setting of trees with recursive lengths with the length metric $\distgamma$.

We will need the following elementary lemma.
\begin{lemma}
  Let $f$ and $g$ be two positive non-increasing functions defined on
  $\intoo{0}{1}$. Let $(r_n)$ be a decreasing sequence of positive numbers
  converging to $0$. Assume that
  \[
    \frac{f(r_n)}{g(r_n)}
    \xrightarrow[n \to \infty]{} \ell \in \intfo{0}{\infty}
    \text{ and }
    \frac{f(r_{n+1})}{f(r_n)}
    \xrightarrow[n \to \infty]{} 1.
  \]
  Then, we have $\lim_{r \downarrow 0} f(r)/g(r) = \ell$.
\end{lemma}
\begin{proof}
  Let $\varepsilon > 0$ and $n_0$ be large enough so that for all $n \geq n_0$,
  \[
    \frac{f(r_n)}{g(r_n)}\frac{f(r_{n+1})}{f(r_n)} \leq \ell + \varepsilon
    \text{ and }
    \frac{f(r_{n+1})}{g(r_{n+1})}\frac{f(r_{n})}{f(r_{n+1})}
    \geq \ell - \varepsilon.
  \]
  Then, using the assumption that $(r_n)$ is decreasing to $0$,
  for all $r \leq r_{n_0}$, there exists $n \geq n_0$ such that
  $r_{n+1} < r \leq r_n$ and we have
  \[
    \ell - \varepsilon
    \leq
    \frac{f(r_{n+1})}{g(r_{n+1})}\frac{f(r_{n})}{f(r_{n+1})}
    = \frac{f(r_n)}{g(r_{n+1})}
    \leq \frac{f(r)}{g(r)}
    \leq \frac{f(r_{n+1})}{g(r_n)}
    =\frac{f(r_n)}{g(r_n)}\frac{f(r_{n+1})}{f(r_n)} \leq \ell +
    \varepsilon.\qedhere
  \]
\end{proof}
\begin{proposition}[dimension of a flow rule]\label{prop:holdergamma}
  Let $\Theta$ be a $\GW$-flow rule such that there exists a
  $\Theta$-invariant probability measure $\mu$
  which is absolutely continuous with respect
  to $\GW$.
  Then, almost surely, the probability measure
  $\Theta_T$ is exact-dimensional on the metric space
  $(\rays T, \distgamma)$,
  with deterministic dimension
  \begin{equation}\label{eq:hausdgamma}
    \dim^{\distgamma} \Theta_T
    =
    \frac{
      \dim^{\distrays} \Theta_T}{
      \E_\mu [ -\log (1 - \Gamma_\root^{-1})]
    }.
  \end{equation}
\end{proposition}
\begin{proof}
  We first prove that, for $\GW$-almost every tree $t$, for $\Theta_t$-almost
  every ray $\xi$,
  \begin{equation}
    \label{eq:dimflowlimn}
    \lim_{n \to \infty}
    \frac{
      -\log \Theta_t (\xi_n)}{
      -\log (1 - \abs{\xi_n}^\Gamma)
    }
    =
    \frac{
      \dim^{\distrays} \Theta_T}{
      \E_\mu [ -\log (1 - \Gamma_\root^{-1})]
    }.
  \end{equation}
  The numerator equals
  \[
    \sum_{k=0}^{n-1} -
    \log \frac{\Theta_t \left(\xi_{k+1} \right)}{
      \Theta_t \left( \xi_k \right)},
  \]
  so, by the ergodic theorem (for non-negative functions),
  recalling that $\mu\ltimes\Theta$ is ergodic and $\mu$ is equivalent to
  $\GW$, for $\GW$-almost every $t$ and $\Theta_t$-almost every $\xi$,
  \begin{equation}\label{eq:cvdimdistrays}
    \frac{1}{n}
    \sum_{k=0}^{n-1}
    - \log \frac{\Theta_t \left(\xi_{k+1} \right)}{%
        \Theta_t \left( \xi_k \right)
      }
    \xrightarrow[ n \to \infty ]{}
    \E_\mu [-\log \Theta_T \left( \Xi_1 \right)]
    = \dim^{\distrays} \Theta_T \in \intof{0}{\log m}.
  \end{equation}
  For the denominator, we have, for any $\xi$ in $\rays t$ and any $n \geq 1$,
  \[
    \frac{1}{n+1} (-\log) (1 - \abs{\xi_n}^\Gamma )
    = \frac{1}{n+1}
    \sum_{i=0}^{n} -\log (1 - \gamma_t(\xi_i)^{-1}).
  \]
  Again by the pointwise ergodic theorem, we have, for $\GW$-almost every $t$ and
  $\Theta_t$-almost every $\xi$,
  \begin{equation}\label{eq:cvlogr}
    \frac{1}{n+1} (-\log) \left(1 - \abs{\xi_n}^\Gamma \right)
    \xrightarrow[n \to \infty]{}
    \E_\mu [ - \log (1 - \Gamma_\root^{-1})]
    \in \intof{0}{\infty}.
  \end{equation}
  Thus, the convergence \eqref{eq:dimflowlimn} is proved.

  Now we want to show the following (\textit{a priori} stronger) statement:
  for $\GW$-almost every $t$ and
  $\Theta_t$-almost every $\xi$,
  \begin{equation}
    \label{eq:dimflowlimr}
    \lim_{r \downarrow 0}
    \frac{
      -\log \Theta_t \ball\left(\xi, r \right)}{
      -\log r
    }
    =
    \frac{
      \dim^{\distrays} \Theta_T}{
      \E_\mu [ -\log (1 - \Gamma_\root^{-1})]
    },
  \end{equation}
  where $\ball\left(\xi, r\right)$ is the closed ball of center $\xi$ and radius
  $r$ in the metric space $(\rays t, \distgamma)$.

  Let $t$ be a marked tree and $\xi$ be a ray in $t$ such that
  \eqref{eq:cvdimdistrays} and \eqref{eq:cvlogr} hold.
  Denote, for $n \geq 0$, $r_n = 1 - \abs{\xi_n}^\Gamma$.
  The sequence $(r_n)$ is positive, decreasing, and converges to $0$ by
  \eqref{eq:cvlogr}.
  For $r$ in $\intoo{0}{1}$, define $f(r) = -\log \Theta_t\ball(\xi,r)$ and
  $g(r) = -\log(r)$.
  The functions $f$ and $g$ are positive and non-increasing.
  Furthermore,
  \[
    \frac{f(r_{n+1})}{f(r_n)}
    =
    \frac{-\log\Theta_t(\xi_{n+1})}{-\log\Theta_t(\xi_n)}
    =
    1 +
    \frac{
      -\log
      \frac{
        \Theta_t(\xi_{n+1})}{
        \Theta_t(\xi_n)
      }
    }{
      -\log\Theta_t(\xi_n)
    }
    =
    1 +
    \frac{-\frac{1}{n}
      \log \frac{\Theta_t(\xi_{n+1})}{\Theta_t(\xi_n)}
    }{
      -\frac{1}{n}\log\Theta_t(\xi_n)
    }.
  \]
  Using \eqref{eq:cvdimdistrays}, we obtain
  $\lim_{n \to\infty} f(r_{n+1})/f(r_n) = 1$, and conclude by the preceding lemma.
\end{proof}

We now associate to the random marked tree $T$ an \emph{age-dependent process}
(in the definition of \cite[chapter~4]{athreya_ney_book}).
For any $x \in T$, let
$\Lambda_x \coloneqq -\log \left(1 - \Gamma_x^{-1}\right)$
be the lifetime of particle $x$.
Informally, the root lives for time
$\Lambda_\root$, then simultaneously dies
and gives birth to $\nu_T(\root)$ children who all have i.i.d.
lifetimes $\Lambda_1$, $\Lambda_2$, \dots, $\Lambda_{\nu_T(\root)}$, and then
independently live and produce their own offspring and die, and so on.
We are interested in the number $Z_u(T)$ of living individuals at time $u>0$,
that is
\[
  Z_u (T) \coloneqq \#
  \setof[\Big]{x \in T}{
    \sum_{y \preceq \prt{x}}^{} \Lambda_y < u
    \leq \sum_{y \preceq x}^{} \Lambda_y
  }.
\]
The \emph{Malthusian parameter} of this process is the unique real number
$\alpha > 0$ such that
\begin{equation}\label{eq:malthus}
  \E \big[ e^{ - \alpha \Lambda_\root } \big] = \frac{1}{m}.
\end{equation}

We now assume that $ \sum_{k=1}^{\infty}p_k k\log k$ is finite.
Under this assumption, we know from
\cite[Theorem~5.3]{jagersnerman_growth}\footnote{See also Section~3.4 of the
preliminary Saint-Flour 2017 lecture notes by Remco~Van~Der~Hoffstad.} that
there exists a positive random variable $W^\Gamma (T)$ such that
$\E [W^\Gamma (T) ] = 1$ and almost surely,
\begin{equation}\label{eq:defw}
  \lim_{u \to \infty} e^{-u\alpha}Z_u (T)
  = W^\Gamma (T).
\end{equation}
By definition, we obtain the recursive equation
\begin{equation}\label{eq:recunifgamma}
  W^\Gamma (T)
  =
  e^{-\alpha \Lambda_\root}
  \sum_{j=1}^{\numch_T(\root)} W^\Gamma (T[i]).
\end{equation}

We now go back to our original tree with recursive lengths.
Equations~\eqref{eq:malthus}, \eqref{eq:defw} and \eqref{eq:recunifgamma}
become
\begin{gather}
  \E [ (1 - \Gamma_\root^{-1})^\alpha ] = 1/m ;
  \\
  \lim_{\varepsilon \to 0}\varepsilon^\alpha Z_{-\log(\varepsilon)} (T)
    = W^\Gamma (T);
  \label{eq:limZ}\\
  W^\Gamma (T) = (1 - \Gamma_\root^{-1} )^\alpha
    \sum_{j=1}^{\numch_T(\root)}
    W^\Gamma (T[i] ).
  \label{eq:recunifgamma2}
\end{gather}
We define the $\GW$-flow rule $\unif^\Gamma$ by
\[
  \unif^\Gamma_T (i)
  =
  \frac{
    W^\Gamma (T [i] )}{
    \sum_{j=1}^{\numch_T(\root)} W^\Gamma (T[j])
  },
  \quad \forall \, 1 \leq i \leq \numch_T(\root).
\]
\begin{proposition}[dimension of the limit uniform
  measure]\label{prop:dimunifgamma}
  Assume that
  $ \sum_{k \geq 1}^{} p_k k\log k$ is
  finite. Then, both
  the dimension of $\unif^\Gamma_T$ and the Hausdorff dimension of the boundary
  $\rays{T}$, with respect to the metric
  $d^\Gamma$, are almost surely equal to the Malthusian parameter $\alpha$.
\end{proposition}
\begin{proof}
  We can use \Cref{thm:alg}, with $h(u,v) = uv$ and the marks
  equal to $( (1-\Gamma_x^{-1})^\alpha )_{x \in T}$
  (or a direct computation) to show that the probability measure with density
  $W^\Gamma$ with respect to $\GW$ is $\unif^\Gamma$-invariant.
  So we may apply \Cref{prop:holdergamma} to obtain that the
  dimension of $\unif^\Gamma$ with respect to the metric $\distgamma$ equals
  \[
    \dim^{d^\Gamma} \unif^\Gamma_T
    =
    \frac{
      \dim^{\distrays} \unif^\Gamma_T }{
      \E [-\log (1 - \Gamma_\root^{-1})W^\Gamma (T) ]
    }.
  \]
  The numerator equals, by formula~\eqref{cor:hausdflow} and
  the recursive equation~\eqref{eq:recunifgamma2},
  \begin{align*}
    \dim^{\distrays} \unif^\Gamma
    &=
    \E \Bigl[
      \Bigl(
        -\log W^\Gamma (T [\Xi_1] )
        + \log  \sum_{j=1}^{\nu_T(\root)} W^\Gamma \left(T[j]\right)
      \Bigr)
      W^\Gamma(T)
    \Bigr]
    \\
    &=
    \E \bracks*{
      \pars*{
        \log((1 - \Gamma_\root^{-1})^{-\alpha}) +
        \log W^\Gamma(T) -
        \log W^\Gamma( T[\Xi_1] )
      }
      W^\Gamma(T)
    }
    \\
    &=
    \alpha \E \bracks*{ -\log (1-\Gamma_\root^{-1})W^\Gamma(T) },
  \end{align*}
  provided we can show that the term
  $ (\log W^\Gamma(T) - \log W^\Gamma (T [\Xi_1] ))W^\Gamma(T)$
  is bounded from below by an integrable random variable.

  To prove this, we first use the recursive
  equation~\eqref{eq:recunifgamma2}, to obtain
  \[
    \log W^\Gamma(T) - \log W^\Gamma (T [\Xi_1] )
    =
    \alpha\log (1 - \Gamma_\root^{-1})
    + \log \pars*{
      \frac{
        \sum_{i=1}^{\numch_t{\root}} W^\Gamma(T[i])}{
        W^\Gamma(T[\Xi_1])
      }
    }.
  \]
  Since $\Xi_1$ is one of the children of the root, we have
  \[
    \log \pars*{
      \frac{
        \sum_{i=1}^{\nu_t{\root}}
        W^\Gamma(T[i])}{W^\Gamma(T[\Xi_1])
      }
    }
    \geq 0.
  \]
  Using again \eqref{eq:recunifgamma2}, we obtain
  \begin{align*}
    &\left(\log W^\Gamma(T) - \log W^\Gamma (T [\Xi_1] )\right)W^\Gamma(T)
    \\
    &\geq \alpha\log (1 - \Gamma_\root^{-1})
      (1-\Gamma_\root^{-1})^\alpha
      \sum_{i=1}^{\numch_T(\root)}W^\Gamma(T[i])
      \geq - \frac{1}{e} \sum_{i=1}^{\numch_T(\root)}W^\Gamma(T[i]),
  \end{align*}
  where, for the last inequality, we have used the fact that the minimum
  of the function $x \mapsto  x^\alpha \log(x)$ on the interval
  $\intoo{0}{1}$ is $-1 / (\alpha e)$.
  Since
  $
    \E [\sum_{i=1}^{\nu_T(\root)}W^\Gamma(T[i])] = m < \infty
  $,
  this concludes the proof that
  $
    \dim^{\distrays} \unif^\Gamma
    = \alpha \E [ -\log (1-\Gamma_\root^{-1})W^\Gamma(T) ]
  $.

  We remark that
  $
    \E [\log (1 - \Gamma_\root^{-1})W^\Gamma(T)]
  $
  is finite, because
  $
    \dim^{\distrays}\unif^\Gamma \leq \log m
  $.
  Thus we have
  \[
    \dim^{d^\Gamma} \unif^\Gamma_T
    =
    \frac{
      \dim^{\distrays} \unif^\Gamma_T }{
      \E [-\log (1 - \Gamma_\root^{-1})W^\Gamma (T) ]
    }
    =
    \frac{
      \alpha\E [\log (1 - \Gamma_\root^{-1})W(T)] }{
      \E [-\log (1 - \Gamma_\root^{-1})W^\Gamma (T) ]
    }
    =\alpha.
  \]

  We now know that the Hausdorff dimension of the boundary $\rays{T}$ (with
  respect to $d^\Gamma$) is almost
  surely greater or equal to $\alpha$, so we just need the upper bound.
  Recall the definition (\eqref{eq:hausmdeltas} and \eqref{eq:hausms}) of
  the Hausdorff measures.
  We let
  \[
    A_\varepsilon
    \coloneqq
    \setof{x \in T}{ 1 - \abs{x}^\Gamma \leq \varepsilon < 1 -
      \abs{\prt{x}}^\Gamma },
  \]
  whose number of elements is $Z_{-\log(\varepsilon)}(T)$.
  By the limit~\eqref{eq:limZ}, we have
  \[
    \hausm_\varepsilon^\alpha (\rays{T})
    \leq
    \sum_{x \in A_\varepsilon}^{} (\diam^\Gamma \cyl[T]{x})^\alpha
    \leq \varepsilon^\alpha Z_{-\log(\varepsilon)}(T)
    \xrightarrow[ \varepsilon \to 0]{ \text{a.s.} } W^\Gamma (T),
  \]
  so
  $\hausm^\alpha (\rays{T}) \leq W^\Gamma(T) < \infty$,
  which concludes the proof.
\end{proof}
\begin{proposition}[dimension drop for other flow rules]
  Assume that $\sum_{k \geq 1}^{} p_k k \log k$ is finite.
  Let $T$ be a $(\Gamma,\pp)$-Galton-Watson tree and $\Theta$ be a
  $\GW$-flow rule such that $\Theta_T$ and $\unif^\Gamma_T$ are not
  almost surely equal and there exists a $\Theta$-invariant
  probability measure $\mu$ absolutely continuous with respect to $\GW$.
  Then the dimension of $\Theta$ with respect to the distance $\distgamma$
  is almost surely strictly less than the Malthusian parameter $\alpha$.
\end{proposition}
\begin{proof}
  First, we remark that if
  $\E_\mu [-\log (1 - \Gamma_\root^{-1})]$
  is infinite, then the
  Hausdorff dimension of $\Theta_T$ with respect to the distance $\distgamma$
  is almost surely equal to $0$, so there is nothing to prove.

  So we assume that
  $\E_\mu [-\log (1 - \Gamma_\root^{-1})]$ is
  finite.
  Let $\Xi$ be a random ray in $\rays{T}$ with distribution $\Theta_T$. Using
  formula~\eqref{cor:hausdflow} and conditioning on the value of $\Xi_1$ gives
  \[
    \dim^{\distrays} (\Theta_T)
    =
    \E_\mu \Bigl[
      \sum_{i=1}^{\numch_T(\root)} -\Theta_T(i)\log(\Theta_T(i))
    \Bigr]
    <
    \E_\mu\Bigl[
      \sum_{i=1}^{\numch_T(\root)} -\Theta_T(i) \log \unif^\Gamma_T (i)
    \Bigr],
  \]
  where, for the strict inequality we have used Shannon's inequality
  together with the fact (\Cref{prop:differentflow}) that almost
  surely, $\Theta_T$ is different from $\unif^\Gamma_T$.
  This upper bound is equal to
  \[
    \E_\mu \bracks*{
      \alpha (-\log (1 - \Gamma_\root^{-1}))
      + \log W^\Gamma (T)
      - \log W^\Gamma (T [\Xi_1] )
    }.
  \]
  Once again, all that is  left to prove is that the last two terms are bounded
  from below by an integrable random variable, and this is the case, because
  \[
    \log \frac{W^\Gamma(T)}{W^\Gamma (T[\Xi_1] )}
    \geq
    \alpha \log (1 - \Gamma_\root^{-1}),
  \]
  and by our assumption that
  $\E_\mu [\log(1 - \Gamma_\root^{-1})]$
  is finite.
  Cancelling out this term in
  equation~\eqref{eq:hausdgamma}, we finally obtain
  $\dim^{d^\Gamma} \Theta_T < \alpha$.
\end{proof}
Before we state and prove the main theorem of this subsection, we want to know
when the dimension (with respect to $\distgamma$)
of the harmonic measure equals $0$.
\begin{lemma}
  \label{lem:zerodim}
  Let $T$ be a $(\Gamma,\pp)$-Galton-Watson marked tree.
  Assume that
  $\E [\kappa (\phi(T))]$ and
  $ \sum_{k \geq 1}^{}p_k k$
  are finite.
  Then, we have
  \[
    \E [ \log (1 - \Gamma_\root^{-1} )\kappa(\phi(T)) ]
    < \infty
    \iff
    \E [ \log (1 - \Gamma_\root^{-1})] < \infty.
  \]
\end{lemma}
\begin{proof}
  By Tonelli's theorem, the definition of $\kappa$, and the associativity
  property of the function $h$, we have
  \[
    \E [ \log (1 - \Gamma_\root^{-1} )\kappa(\phi(T)) ]
    = \E \bracks[\Bigg]{
        \log (1 - \Gamma_\root^{-1} )
        h \Bigl(
          \Gamma_\root,
          h\bigl(
            \sum^{\nu_T(\root)}_{i=1} \phi(T[i]),
            \sum^{\nu_{\tilde{T}}(\root)}_{j=1} \phi (\tilde{T}[j])
          \bigr)
        \Bigr)
      },
  \]
  where $\tilde{T}$ is a $(\Gamma, \pp)$-Galton-Watson tree,
  independent of $T$.
  Since for any $u$ and $v$ greater than $1$, $h(u,v) > 1$, the direct
  implication is proved.
  For the reciprocal implication,
  recall that for $u$ and $v$ in $\intoo{1}{\infty}$,
  $h(u,v) < v$, hence
  \[
    h \Bigl(
      \Gamma_\root,
      h \bigl(
        \sum^{\nu_T(\root)}_{i=1} \phi(T[i]),
        \sum^{\nu_{\tilde{T}}(\root)}_{j=1} \phi (\tilde{T}[j])
      \bigr)
    \Bigr)
    <
    h \Bigl(
      \sum^{\numch_T(\root)}_{i=1} \phi(T[i]),
      \sum^{\numch_{\tilde{T}}(\root)}_{j=1} \phi (\tilde{T}[i])
    \Bigr).
  \]
  The right-hand side of the previous inequality is integrable.
  Indeed,
  \begin{align*}
    h \Bigl(
      \sum^{\numch_T(\root)}_{i=1} \phi(T[i]),
      \sum^{\numch_{\tilde{T}}(\root)}_{j=1} \phi (\tilde{T}[i])
    \Bigr)
    &=
    \frac{
      \bigl( \sum^{\numch_T(\root)}_{i=1} \phi (T[i]) \bigr)
      \bigr( \sum^{\numch_{\tilde{T}}(\root)}_{j=1} \phi (\tilde{T}[j]) \bigr)
    }{%
      \sum^{\numch_T(\root)}_{i=1} \phi (T[i]) +
      \sum^{\numch_{\tilde{T}}(\root)}_{j=1} \phi (\tilde{T}[j])
      - 1
    }
    \\
    &\leq
    \sum^{\numch_T(\root)}_{i=1}
    \frac{
      \phi (T[i])
      \bigl(
        \sum^{\numch_{\tilde{T}}(\root)}_{j=1} \phi (\tilde{T}[j])
      \bigr)
    }{%
    \phi (T[i]) +
      \sum^{\numch_{\tilde{T}}(\root)}_{j=1} \phi (\tilde{T}[j]) - 1
    }
  \end{align*}
  and the expectation of this upper bound equals, by independence,
  \[
    \E [\nu_T(\root)] \E [\kappa (\phi(T))],
  \]
  which is finite by assumption.
  Thus, using the fact that
  \[
    \Gamma_\root \text{ and }
    h \Bigl(
      \sum^{\nu_T(\root)}_{i=1} \phi(T[i]),
      \sum^{\nu_{\tilde{T}}(\root)}_{j=1} \phi (\tilde{T}[i])
    \Bigr)
    \text{ are independent},
  \]
  we have
  \[
    \E [ \log (1 - \Gamma_\root^{-1} )\kappa(\phi(T)) ]
    \leq
    \E [ \log ( 1 - \Gamma_\root^{-1} ) ]
    \E [\nu_T(\root)] \E [\kappa (\phi(T))],
  \]
  which proves the reciprocal implication of the lemma.
\end{proof}
Putting everything together, we finally obtain the dimension drop for
the flow rule $\harm^\Gamma$, with respect to the distance $\distgamma$.

\begin{theorem}
  Let $T$ be a $(\Gamma,\pp)$-Galton-Watson marked tree, with metric
  $\distgamma$ on its boundary.
  Assume that both
  $\E [\kappa (\phi(T))]$ and
  $ \sum_{k \geq 1}^{}p_k k \log k$
  are finite.
  Then, almost surely, the flow
  $\harm_T^\Gamma$ is exact-dimensional, of deterministic dimension
  \begin{equation}\label{eq:dim_gamma}
    \dim^{d^\Gamma} \harm^\Gamma (T)
    =
    \frac{
      \E [\log (1 - \Gamma_\root^{-1}\phi(T))
      \kappa(\phi(T))] }{
      \E [\log (1 - \Gamma_\root^{-1})
      \kappa(\phi(T))] }
      -1,
  \end{equation}
  except in the case
  $\E [-\log (1 - \Gamma_\root^{-1})] = \infty$,
  where it is $0$.
  This deterministic dimension is strictly less
  than the Malthusian parameter $\alpha$
  (which is almost surely the Hausdorff dimension of the boundary $\rays{T}$
  with respect to the distance $\distgamma$) as soon as the mark law and the
  reproduction law are not both degenerated.
\end{theorem}
\begin{proof}
  The formula for the Hausdorff dimension is just a rewriting using
  equations~\eqref{eq:hausdgamma} and~\eqref{eq:dimgamma}.
  All that is left to prove is that if
  there exists a positive real number $K$, such
  that, for $\GW$-almost every tree $t$,
  $W^\Gamma (t) = K \times \phi (t)$, then both the mark law and
  the reproduction law are degenerated.

  We assume that the latter assertion holds,
  and we proceed similarly as in the proof of \Cref{thm:reclengthd}.
  From the recursive equation~\eqref{eq:recunifgamma2}
  for $W^\Gamma$, we deduce that almost surely
  \begin{equation}\label{eq:degenerated}
    \sum_{i=1}^{\nu_T(\root)} \phi (T[i])
    =
    (1 - \Gamma_\root^{-1})^{-\alpha} \phi(T),
  \end{equation}
  and plugging it into the recursive equation~\eqref{eq:recbetagamma}
  for $\phi$, we obtain that,
  almost surely,
  \[
    \phi(T) = \Gamma_\root
    (1 - (1 - \Gamma_\root^{-1})^{\alpha + 1}).
  \]
  This implies that each $\phi(T[i])$ depends only on $\Gamma_i$ and
  \[
    \sum_{i = 1}^{\nu_T(\root)} \phi (T [i])
    =
    \Gamma_\root (1 - \Gamma_\root^{-1})^{-\alpha} + 1 - \Gamma_\root,
  \]
  so, by independence, $ \sum_{i=1}^{\nu_T(\root)} \phi (T[i])$ must
  be constant, which imposes that $\nu_T(\root)$ must be constant (equal to $m$)
  and that
  the law of $\phi(T)$ is degenerated.
  From~\eqref{eq:degenerated}, we now see that
  this implies that
  $(1 - \Gamma_\root^{-1})^\alpha = 1 /m$
  and the law of $\Gamma_\root$ is degenerated.
\end{proof}
To conclude this work, we want to check that our formula
\eqref{eq:dim_gamma} is consistent with the
formula obtained in \cite{Curien_LeGall_harmonic}.
From now on, we work under the following hypotheses:
\begin{enumerate}
  \item the reproduction law is given by $p_2 = 1$;
  \item the common law of the marks is the law of $U^{-1}$, where the law of
    $U$ is uniform on $\intoo{0}{1}$.
\end{enumerate}
\begin{remark}
  The function denoted by $t \mapsto \kappa(\phi(t))$, in
  \cite[Proposition~25]{Curien_LeGall_harmonic} is slightly different (it
  differs by a factor $1/2$) from our function also denoted by $\kappa(\phi(t))$.
\end{remark}
Under these hypotheses, Curien and Le Gall proved that
the dimension (with respect to the metric
$\distgamma$) of the harmonic measure is almost surely (see
\cite[Proposition~4]{Curien_LeGall_harmonic}):
\begin{equation}\label{eq:dim_curien_legall}
  \dim^{\distgamma} \harm^\Gamma (T)
  =
  2 \E \bracks*{
      \log\pars*{
        \frac{\phi_1+\phi_2}{\phi_1}
      }
      \frac{\phi_1 \tilde{\phi}}{\tilde{\phi} + \phi_1 + \phi_2 - 1}
    }
    \bigg/
    \E \bracks*{
      \frac{\phi_1\phi_2}{\phi_1 + \phi_2 - 1}
    },
\end{equation}
where $\phi_1$, $\phi_2$ and $\tilde{\phi}$ are independent copies of $\phi(T)$.
For short, we write $U = \Gamma_\root^{-1}$, $\phi = \phi(T)$,
$\phi_1 = \phi(T[1])$ and $\phi_2 = \phi(T[2])$.

We first show that
\begin{equation}
  \label{eq:checknum}
  \E \bracks*{
    -\log \pars*{ \frac{1 - U\phi}{1 - U} }
    \kappa(\phi)
  }
  =
  2 \E \bracks*{
    \log \pars*{
      \frac{\phi_1+\phi_2}{\phi_1}
    }
    \frac{\phi_1 \tilde{\phi}}{\tilde{\phi} + \phi_1 + \phi_2 - 1}
  }.
\end{equation}
Recall from the proof of \cref{thm:reclengthd} that, by stationarity,
\begin{equation*}
  \E [ \log (\phi) \kappa (\phi) ]
  =
  \E [ \log (\phi ( T [ \Xi_1 ] ) ) \kappa (\phi) ].
\end{equation*}
By the recursive formula \eqref{eq:recphi},
\[
  \frac{1 - U\phi}{1-U}
  =
  \frac{U^{-1}}{\phi_1 + \phi_2 + U^{-1} -1}
  =
  \frac{\phi}{\phi_1 + \phi_2},
\]
thus we obtain
\begin{align*}
  &\E \bracks*{
    \log \pars*{ \frac{1 - U\phi}{1 - U} } \kappa(\phi)
  }
  =
  \E \bracks*{
    \log \pars*{\frac{\phi}{\phi_1 + \phi_2}} \kappa(\phi)
  }
  \\
  &=
  \E \bracks*{
    \log \pars*{\frac{\phi(T[\Xi_1])}{\phi_1 + \phi_2}} \kappa(\phi)
  }
  \\
  &= \E \bracks*{
    \frac{\phi_1}{\phi_1 + \phi_2}
    \log \pars*{\frac{\phi_1}{\phi_1 + \phi_2}}
    \kappa (\phi)
    +
    \frac{\phi_2}{\phi_1 + \phi_2}
    \log \pars*{\frac{\phi_2}{\phi_1 + \phi_2}}
    \kappa (\phi)
  }
  \\
  &= 2 \E \bracks*{
  \frac{\phi_1}{\phi_1 + \phi_2}
  \log \pars*{\frac{\phi_1}{\phi_1+\phi_2}} \kappa (\phi)
  },
\end{align*}
by symmetry.
Let $\tilde{T}$ be a $(\Gamma,\pp)$-Galton-Watson tree such that
the mark of the root is $U^{-1}$, and $\tilde{T}[1]$ and $\tilde{T}[2]$ are
independent of $T[1]$ and $T[2]$. Write $\tilde{\phi}$ for the conductance of
$\tilde{T}$ and $\tilde{\phi}_i = \phi (\tilde{T}[i])$ for $i = 1, 2$.
By Tonelli's theorem and the definition of $\kappa$, we have
\begin{align*}
  \E \bracks*{
    \log \pars*{ \frac{1 - U\phi}{1 - U} }
      \kappa(\phi)
  }
  &= 2 \E \bracks*{
    \frac{\phi_1}{\phi_1 + \phi_2}
    \log \pars*{\frac{\phi_1}{\phi_1+\phi_2}}
    h \pars*{
      h \pars*{ U^{-1}, \phi_1 + \phi_2 },
      \tilde{\phi}_1 + \tilde{\phi}_2
    }
  }
  \\
  &= 2 \E \bracks*{
    \frac{\phi_1}{\phi_1 + \phi_2}
    \log \pars*{\frac{\phi_1}{\phi_1+\phi_2}}
    h \pars*{
      h \pars*{ U^{-1}, \tilde{\phi}_1 + \tilde{\phi}_2 },
      \phi_1 + \phi_2
    }
  }
  \\
  &= 2 \E \bracks*{
    \frac{\phi_1}{\phi_1 + \phi_2}
    \log \pars*{ \frac{\phi_1}{\phi_1+\phi_2} }
    h \pars*{ \tilde{\phi}, \phi_1 + \phi_2 }
  }
  \\
  &= 2 \E \bracks*{
    \frac{\phi_1}{\phi_1 + \phi_2}
    \log \pars*{ \frac{\phi_1}{\phi_1+\phi_2} }
    \frac{
      \tilde{\phi}(\phi_1+\phi_2)}{
      \tilde\phi + \phi_1 + \phi_2 - 1
    }
  },
\end{align*}
where, between the first and the second line, we have used the associativity and
the symmetry of the function $h$.
The proof of \eqref{eq:checknum} is complete.

Now, we want to show that
\begin{equation}\label{eq:checkdenum}
  \E \bracks*{ -\log (1 - U) \kappa(\phi) }
  =
  \E \bracks*{ \frac{\phi_1 \phi_2}{\phi_1 + \phi_2 - 1}}.
\end{equation}
Here, we rely heavily on the fact that $U$ is uniform on $\intoo{0}{1}$.
From~\cite[equation~(13)]{Curien_LeGall_harmonic}, we know that, for any
function
$g : \intfo{1}{\infty} \to \R_+$ such that $g(x)$ and $g'(x)$ are both $o(x^a)$ for some
$a$ in $\intoo{0}{\infty}$, we have
\begin{equation}\label{eq:curien_legall_identity}
  \E [ g(\phi_1 + \phi_2) ] = \E [\phi_1(\phi_1 - 1) g'(\phi_1)]
  + \E [g(\phi_1)].
\end{equation}
As before, let $\phi_1$, $\phi_2$, $\tilde{\phi}_1$ and $\tilde{\phi}_2$ be
independent copies of $\phi(T)$, independent of $U$.
Let $\psi_1 : \intoo{1}{\infty}^3 \to \intoo{1}{\infty}$ be defined by
\[
  \psi_1 (x, y, z) = h(x,h(y,z)) = \frac{xyz}{xy + yz + xz - x - y - z +1}.
\]
By definition of $\kappa$, we have
\[
  \E \bracks*{ -\log (1 - U) \kappa(\phi) }
  =
  \E \bracks*{
    -\log(1-U)
    \psi_1 \pars*{ U^{-1}, \phi_1 + \phi_2, \tilde{\phi}_1 + \tilde{\phi_2} }
  }.
\]
For $x$, $y$, $z$ in $\intoo{1}{\infty}$, let
\[
  \psi_2(x,y,z)
  = \psi_1(x,y,z) + y(y-1) \partial_y \psi_1 (x,y,z)
  = \frac{x^2y^2z^2}{\left(xy+xz+yz-x-y-z+1\right)^2}.
\]
Reason conditionally on $U$, $\tilde{\phi}_1$ and $\tilde{\phi}_2$ and apply
the identity \eqref{eq:curien_legall_identity} to the function
$y \mapsto \psi_1(x,y,z)$, to obtain
\[
  \E \bracks*{ -\log(1-U) \kappa(\phi) }
  =
  \E \bracks*{
    -\log(1-U)
    \psi_2 \pars*{U^{-1}, \phi_1, \tilde{\phi}_1 + \tilde{\phi}_2}
  }.
\]
Playing the same game again, we obtain
\[
  \E \bracks*{ -\log(1-U) \kappa(\phi) }
  =
  \E \bracks*{
    -\log(1-U)
    \psi_3 \pars*{ U^{-1}, \phi_1, \tilde{\phi}_1 }
  },
\]
with the function $\psi_3$ defined by
\begin{align*}
  &\psi_3 (x,y,z) = \psi_2 (x,y,z) + z(z-1) \partial_z \psi_2(x, y, z)
  \\
  &= (xyz)^2
  \left[
    \frac{2xyz}{\left(xy + xz + yz - x - y - z + 1\right)^3}
    -
    \frac{1}{\left(xy + xz + yz - x - y - z + 1\right)^2}
  \right].
\end{align*}
Fix $y$ and $z$ in $\intoo{1}{\infty}$ and let, for $u$ in $\intoo{0}{1}$,
\begin{align*}
  &\psi_4(u) = \psi_3(u^{-1},y,z)
  \\
  &= y^2z^2
  \left[
    \frac{2yz}{\left[(yz+1-y-z)u + (y+z-1)\right]^3}
    - \frac{1}{\left[ (yz + 1 - y - z)u + (y+z-1)\right]^2}
  \right]
  \\
  &= (a+b)^2
  \left[ \frac{2(a+b)}{(au+b)^3} - \frac{1}{(au+b)^2} \right],
\end{align*}
with $a = (yz+1-y-z)$ and $b = (y+z-1)$.
Finally, integrating by parts gives
\[
  \int_0^1 -\log(1-u)\psi_4(u) \dd u = \frac{a+b}{b} = \frac{yz}{y + z -1},
\]
so that, by independence of $U$, $\phi_1$ and $\tilde{\phi}_1$,
\[
  \E \left[
    -\log(1-U)\psi_3\left(U^{-1}, \phi_1, \tilde{\phi_1}\right)
    \middle| \phi_1, \tilde{\phi}_1
  \right]
  =
  \frac{\phi_1\tilde{\phi}_1}{\phi_1 + \tilde{\phi}_1 - 1},
\]
which completes the proof of \eqref{eq:checkdenum}, and the
verification of the consistency of formula \eqref{eq:dim_curien_legall} with
\eqref{eq:dim_gamma}.

\appendix

\bibliographystyle{amsplain}
\providecommand{\bysame}{\leavevmode\hbox to3em{\hrulefill}\thinspace}
\providecommand{\MR}{\relax\ifhmode\unskip\space\fi MR }
\providecommand{\MRhref}[2]{%
  \href{http://www.ams.org/mathscinet-getitem?mr=#1}{#2}
}
\providecommand{\href}[2]{#2}

\end{document}